\documentclass{amsart}
\usepackage{amsmath}
\usepackage{amstext}
\usepackage{amssymb}
\usepackage{amsthm}
\usepackage[all]{xy} %MK
\swapnumbers
\theoremstyle{plain}%default
\newtheorem{thm}{Theorem}[section]
\newtheorem{lem}[thm]{Lemma}

\newtheorem{prop}[thm]{Proposition}
\newtheorem{cor}[thm]{Corollary}

\theoremstyle{definition}
\newtheorem{defi}[thm]{Definition}

\newtheorem{ex}[thm]{Example}
\newtheorem{exs}[thm]{Examples}

\newtheorem{ntn}[thm]{Notation}
\newtheorem{rmd}[thm]{Reminder}

\newtheorem{disc}[thm]{Discussion}
\theoremstyle{remark}
\newtheorem*{note}{Note}
\newtheorem{rmk}[thm]{Remark}
\newtheorem{rmks}[thm]{Remarks}
\newtheorem{strat}[thm]{Strategy}

\DeclareMathOperator{\Id}{Id} 
\DeclareMathOperator{\Ext}{Ext} \DeclareMathOperator{\Hom}{Hom}
 
 \DeclareMathOperator{\Ker}{Ker}
\DeclareMathOperator{\Ima}{Im}

\DeclareMathOperator{\diag}{diag}

\DeclareMathOperator{\lowlim}{lowlim}
\DeclareMathOperator{\Supp}{Supp} 
 
\DeclareMathOperator{\Ass}{Ass}

 \DeclareMathOperator{\Var}{Var}
 \DeclareMathOperator{\Spec}{Spec}

\DeclareMathOperator{\Coker}{Coker}
\def\Z{\mathbb Z}
\def\N{\mathbb N}

\def\fa{{\mathfrak{a}}}

\def\fb{{\mathfrak{b}}}

\def\fc{{\mathfrak{c}}}
\def\fd{{\mathfrak{d}}}
\def\fg{{\mathfrak{g}}}
\def\fh{{\mathfrak{h}}}

\def\fm{{\mathfrak{m}}}

\def\fn{{\mathfrak{n}}}
\def\fp{{\mathfrak{p}}}

\def\nn{\relax\ifmmode{\mathbb N_{0}}\else$\mathbb N_{0}$\fi}
\def\lra{\longrightarrow}

\begin{document}

\title{Lyubeznik numbers, $F$-modules and modules of generalized fractions}
\author{MORDECHAI KATZMAN}
\address{School of Mathematics and Statistics,
University of Sheffield, Hicks Building, Sheffield S3 7RH, United Kingdom}
\email{M.Katzman@sheffield.ac.uk}
\author{RODNEY Y. SHARP}
\address{School of Mathematics and Statistics,
University of Sheffield, Hicks Building, Sheffield S3 7RH, United Kingdom}
\email{R.Y.Sharp@sheffield.ac.uk}

\thanks{}


\date{\today}


\begin{abstract}

This paper presents an algorithm for calculation of the Lyubeznik numbers of a local ring which is a homomorphic image of a regular local ring $R$ of prime characteristic. The methods used employ Lyubeznik's $F$-modules over $R$, particularly his $F$-finite $F$-modules, and also the modules of generalized fractions of Sharp and Zakeri. It is shown that many modules of generalized fractions over $R$ have natural structures as $F$-modules; these lead to $F$-module structures on certain local cohomology modules over $R$, which are exploited, in conjunction with $F$-module structures on injective $R$-modules that result from work of Huneke and Sharp, to compute Lyubeznik numbers. The resulting algorithm has been implemented in Macaulay2.

\end{abstract}

\maketitle

\setcounter{section}{-1}
\section{\bf Introduction}
\label{in}
The aims of this paper are to study connections between the notions of $F$-module and module of generalized fractions over a regular ring $R$ of prime characteristic $p$, and to use these connections to produce an algorithm for calculation of the Lyubeznik numbers of certain local rings that are homomorphic images of regular local rings of characteristic $p$.

The concept of an \emph{$F$-module} was introduced by Gennady Lyubeznik in his seminal paper \cite{Lyube97}, in which
he showed that particularly simple instances of $F$-modules, namely \emph{$F$-finite $F$-modules},
satisfy strong finiteness conditions, including finiteness of the Bass numbers and set of associated primes.
Furthermore, the methods employed in \cite{Lyube97} are fairly constructive, yielding, for example, algorithms for determining the set of associated primes of an $F$-finite
$F$-module.
We will review the necessary concepts from the theory of $F$-finite $F$-modules in Section \ref{GLF} of this paper.

One motivation for Lyubeznik's work in \cite{Lyube97} was a desire to strengthen certain results of C. Huneke and the second author in \cite{58}; in that paper, Huneke and Sharp proved that, for an ideal $\fa$ of $R$, and for an integer $j\geq 0$, the local cohomology module $H^j_{\fa}(R)$ (which could well fail to be finitely generated) has a finite set of associated primes and finite Bass numbers. Local cohomology theory, due to A. Grothendieck (see \cite{Groth67}), is a powerful tool in algebraic geometry and commutative algebra and the study of local cohomology modules has yielded many insights.

Central to Huneke's and Sharp's argument in \cite{58} was the result that $E \cong F(E)$ for every injective $R$-module $E$, where $F$ is the so-called {\em Frobenius functor} $R' \otimes_R \bullet$, where $R'$ denotes $R$ considered as a left $R$-module in the natural way and as a right $R$-module via the Frobenius homomorphism $f: R \lra R$ (which raises all elements of $R$ to the $p$th power).

Lyubeznik introduced the concept of {\em $F$-module\/} in \cite[1.1]{Lyube97}: an $F$-module is an $R$-module ${\mathcal M}$ equipped with an $R$-module isomorphism $\theta : {\mathcal M} \stackrel{\cong}{\lra} F({\mathcal M})$, called the {\em structure (iso)morphism\/} of ${\mathcal M}$. Lyubeznik's detailed development of his theory of $F$-modules enabled him to prove, among many other things, that if $\fa_1, \ldots, \fa_n$ are ideals of $R$ and $j_1, \ldots, j_n$ are non-negative integers, then the $R$-module $$H^{j_n}_{\fa_n}(H^{j_{n-1}}_{\fa_{n-1}}(\ldots (H^{j_{1}}_{\fa_{1}}(R))\ldots)),$$ obtained by applying $n$ local cohomology functors successively to $R$, has finite set of associated primes and finite Bass numbers.

The second concept at the heart of this paper is that of \emph{module of generalized fractions} introduced by the second author and H. Zakeri in \cite{32} and \cite{34}
as a generalization of the classical concept of fraction formation in commutative algebra. The latter theory produces, for a module $N$ over a commutative ring $A$ and a multiplicatively closed subset $S$ of $A$, a module of fractions $S^{-1}N$.  The theory of generalized fractions produces, for a so-called triangular subset $U$ of $A^n = A \times \cdots \times A$ ($n$ factors), a module of generalized fractions $U^{-n}N$. Some preparatory results about generalized fractions are presented in Section \ref{gf}. Key results for us in this paper are that, over our regular ring $R$ of characteristic $p$, whenever ${\mathcal M}$ is an $F$-module over $R$ and $U$ is a triangular subset of $R^n$, then $U^{-n}{\mathcal M}$ is again an $F$-module in a
naturally-determined way; moreover, if ${\mathcal M}$ is actually an $F$-finite $F$-module (there is a reminder about this concept in Section \ref{GLF}), then so too is $U^{-n}{\mathcal M}$ when $U$ has a simple form determined by a sequence of $n$ elements of $R$. These key results are presented in Sections \ref{gfF} and \ref{mm} respectively. One very useful result proved in Section \ref{mm} is that the tensor product of two $F$-finite $F$-modules over $R$ is again an $F$-finite $F$-module.

In Section \ref{frs}, we use the concept of filter regular sequence in conjunction with a theorem of K. Khashyarmanesh, Sh.\ Salarian and H. Zakeri \cite[Theorem 1.2]{KhaSalZak98} in order to describe many local cohomology modules over $R$ as cohomology modules of complexes of modules of generalized fractions. The final Section \ref{Lnos} applies ideas from the first five sections to produce an algorithm for the calculation of certain so-called `Lyubeznik numbers'.  To explain what these are, we introduce additional notation.

Let $A$ be a $d$-dimensional local ring which can be expressed as a homomorphic image of an $n$-dimensional regular local ring $(S,\fn)$ that contains a subfield, by means of a surjective ring homomorphism $\pi : S \lra A$ having kernel $\fd$. Let $i,j,k \in \nn$. It is known that all the Bass numbers of $H^{k}_{\fd}(S)$ are finite: this was proved by Huneke and Sharp in \cite[Theorem 2.1]{58} in the case of prime characteristic, and by Lyubeznik in \cite[Theorem 3.4]{Lyube93} in the case where $S$ has characteristic $0$.  In \cite[Theorem-Definition 4.1]{Lyube93}, Lyubeznik showed that the Bass number $\mu^{i}(\fn,H^{n-j}_{\fd}(S))$ depends only on $A$, $i$ and $j$, but not on $S$, $n$ or $\pi$; he denoted $\mu^{i}(\fn,H^{n-j}_{\fd}(S))$ by $\lambda_{i,j}(A)$, and subsequently it has become known as the {\em $(i,j)$ Lyubeznik number of $A$}. Recall that this Bass number is equal to the number of copies of $E_S(S/\fn)$, the injective envelope of the simple $S$-module $S/\fn$, that occur in the decomposition, as a direct sum of indecomposable injective $S$-modules, of the $i$th term in the minimal injective resolution of $H^{n-j}_{\fd}(S)$.

In \cite[4.4]{Lyube93}, Lyubeznik noted that
$\lambda_{i,j}(A) = 0$ if $j > d$ or $i > j$, and
that $\lambda_{d,d}(A) \neq 0$.
We can present the Lyubeznik numbers of $A$ in the $(d+1) \times (d+1)$ matrix
\[
\left( \begin{array}{ccc} \lambda_{0,0}(A) & \cdots & \lambda_{0,d}(A)\\& \ddots & \vdots\\ && \lambda_{d,d}(A)\end{array}\right).
\]
The matrix is upper triangular because, as noted above, $\lambda_{i,j}(A) = 0$ if $i > j$. It is usual to omit the $0$s below the main diagonal. This matrix is referred to as the {\em Lyubeznik table of $A$} or the {\em type of $A$}.
We say that the Lyubeznik table is {\em trivial} if $\lambda_{i,j}(A) = 0$ except when $i = j = d$, and $\lambda_{d,d}(A) = 1$. This is the case when $A$ is Cohen--Macaulay and $R$ has prime characteristic: see \cite{NunWitZha16}.%This line and the one above is an %UNCHECKED ADDITION.
The reader should note that the Lyubeznik table of the completion $\widehat{A}$ of $A$ is identical to the Lyubeznik table of $A$.

Lyubeznik numbers sometimes convey information about topological properties. For example, if $A$ as above is complete, equidimensional, has separably closed residue field and has $\dim A \geq 3$, then it is a result of L. N\'u\~nez-Betancourt, S. Spiroff and E. E. Witt in \cite[Theorem 6.1]{NunSpiWit19} that the connectedness dimension (see \cite[19.1.9]{LC}) of $\Spec(A)$ is at least $2$ if and only if $\lambda_{0,1}(A) = \lambda_{1,2}(A) = 0$.

The aim of Section \ref{Lnos} is the presentation of our algorithm for the calculation of the Lyubeznik numbers of certain homomorphic images of regular local rings of prime characteristic. We believe this is the first practical algorithm for the calculation of Lyubeznik numbers in prime characteristic. The algorithm has been implemented in Macaulay2, and any interested reader is referred to http://www.katzman.staff.shef.ac.uk/LyubeznikNumbers/

\section{\bf G. Lyubeznik's $F$-modules}
\label{GLF}

\begin{ntn}
\label{GLF.1} Throughout the paper, we shall assume that $\fa$ is a proper ideal of a regular (commutative Noetherian) ring $R$ of prime characteristic $p$, and we shall use $f :R \lra R$ to denote the Frobenius homomorphism, which raises each element of $R$ to its $p$th power. We use $\nn$ (respectively $\N$) to denote the set of non-negative (respectively positive) integers.

Sometimes we shall wish to work over a commutative ring more general than $R$; we adopt the convention that $A$ will denote a general commutative ring (with identity), and that $A$ will only be assumed to have additional properties, such as being Noetherian, when this is explicitly stated.

Returning to $R$, and following Lyubeznik \cite{Lyube97}, we shall use $F$ to denote the functor $R' \otimes_R \bullet$, where $R'$ is as described in the above Introduction. Thus $F$ is a functor from the category of all $R$-modules and $R$-homomorphisms to itself.  Because $R$ is regular, $F$ is exact. We refer to $F$ as the {\em Frobenius functor.\/}

Lyubeznik introduced the concept of {\em $F$-module\/} in \cite[1.1]{Lyube97}.  An $F$-module is an $R$-module ${\mathcal M}$ equipped with an $R$-module isomorphism $\theta : {\mathcal M} \stackrel{\cong}{\lra} F({\mathcal M})$, called the {\em structure (iso)morphism\/} of ${\mathcal M}$. Huneke--Sharp \cite[Proposition 1.5]{58} shows that every injective $R$-module is an $F$-module, while Lyubeznik \cite[Example 1.2(b)]{Lyube97} shows that every local cohomology module $H^i_{\fa}(R)~(i\in\nn)$ of $R$ is an $F$-module.

Let $M$ be an $R$-module and let $\beta : M \lra F(M)$ be an $R$-homomorphism. We can repeatedly apply $F$ to $\beta$ and obtain an $R$-homomorphism $F^i(\beta) : F^i(M) \lra F^{i+1}(M)$ for each $i \in \N$.  These $F^i(\beta)$ fit together into a commutative diagram
\[
\begin{picture}(300,75)(-150,-25)
\put(-120,40){\makebox(0,0){$                   M
$}}
\put(-60,40){\makebox(0,0){$                    F(M)
$}}
\put(-7,40){\makebox(0,0){$                      \cdots
$}}
\put(52,40){\makebox(0,0){$                     F^i(M)
$}}
\put(120,40){\makebox(0,0){$                    F^{i+1}(M)
$}}
\put(-110,40){\vector(1,0){36}}
\put(-90,44){\makebox(0,0){$^{             \beta
}$}}
\put(-47,40){\vector(1,0){29}}
\put(-34,44){\makebox(0,0){$^{             F(\beta)
}$}}
\put(0,40){\vector(1,0){36}}
\put(20,44){\makebox(0,0){$^{             F^{i-1}(\beta)
}$}}
\put(67,40){\vector(1,0){32}}
\put(83,44){\makebox(0,0){$^{             F^i(\beta)
}$}}
\put(-58,10){\makebox(0,0)[l]{$^{              F(\beta)
}$}}
\put(-120,-20){\makebox(0,0){$               F(M)
$}}
\put(-60,-20){\makebox(0,0){$                F^2(M)
$}}
\put(-7,-20){\makebox(0,0){$                  \cdots
$}}
\put(52,-20){\makebox(0,0){$                    F^{i+1}(M)
$}}
\put(123,-20){\makebox(0,0){$                   F^{i+2}(M)
$}}
\put(-107,-20){\vector(1,0){29}}
\put(-90,-16){\makebox(0,0){$^{             F(\beta)
}$}}
\put(-43,-20){\vector(1,0){24}}
\put(-31,-16){\makebox(0,0){$^{            F^2(\beta)
}$}}
\put(0,-20){\vector(1,0){32}}
\put(74,-20){\vector(1,0){26}}
\put(87,-16){\makebox(0,0){$^{            F^{i+1}(\beta)
}$}}
\put(-120,30){\vector(0,-1){40}}
\put(-116,10){\makebox(0,0){$^{             \beta
}$}}
\put(-60,30){\vector(0,-1){40}}
\put(57,30){\vector(0,-1){40}}
\put(59,10){\makebox(0,0)[l]{$^{              F^{i}(\beta)
}$}}
\put(120,30){\vector(0,-1){40}}
\put(122,10){\makebox(0,0)[l]{$^{              F^{i+1}(\beta)
}$}}
\put(142,40){\vector(1,0){28}}
\put(180,40){\makebox(0,0){$                      \cdots
$}}
\put(144,-20){\vector(1,0){28}}
\put(184,-20){\makebox(0,0){$                      \cdots.
$}}
\end{picture}
\]
The top row in this diagram gives rise to a direct system $(F^i(M))_{i \in \nn}$ (where $F^0$ denotes the identity functor); let ${\mathcal N}$ be the direct limit of this system. Because tensor product commutes with direct limits, there is a natural isomorphism between $F({\mathcal N})$ and the direct limit of the lower row; we use this isomorphism to identify that direct limit with $F({\mathcal N})$. The $(F^i(\beta))_{i \in \nn}$ induce an $R$-isomorphism $\psi : {\mathcal N} \stackrel{\cong}{\lra} F({\mathcal N})$, which therefore makes ${\mathcal N}$ into an $F$-module. We say that $\beta : M \lra F(M)$ is a {\em generating morphism for\/} ${\mathcal N}$. Furthermore, we say that an $F$-module ${\mathcal M}$ is {\em $F$-finite\/} if it has a generating morphism $\alpha : L \lra F(L)$ with $L$ a finitely generated $R$-module. If, in this situation, $\alpha$ is in addition injective, we say that $L$ is a {\em root of ${\mathcal M}$\/} and that $\alpha : L \lra F(L)$ is a {\em root morphism of ${\mathcal M}$\/}. We also refer to the image of $L$ in ${\mathcal M}$ as a {\em root of ${\mathcal M}$}. In yet another variation, we shall say that an $R$-homomorphism $\gamma : G \lra H$ is {\em isomorphic to a root of ${\mathcal M}$} if there exists a root $L$ of ${\mathcal M}$ with root morphism $\alpha : L \lra F(L)$ and isomorphisms $\phi : L \stackrel{\cong}{\lra} G$ and $\psi : F(L) \stackrel{\cong}{\lra} H$ such that the diagram
\[
\begin{picture}(300,65)(-150,-15)
\put(-25,40){\makebox(0,0){$             L $}}
\put(0,43){\makebox(0,0){$^{             \alpha }$}}
\put(20,40){\makebox(0,0)[l]{$             F(L) $}}
\put(-15,40){\vector(1,0){30}} \put(-28,15){\makebox(0,0)[r]{$^{ \phi
}$}} \put(-25,-10){\makebox(0,0){$            G $}}
\put(0,-7){\makebox(0,0){$^{            \gamma }$}}
\put(25,-10){\makebox(0,0)[l]{$             H $}}
\put(-15,-10){\vector(1,0){33}} \put(-25,30){\vector(0,-1){30}}
\put(30,30){\vector(0,-1){30}}
\put(-22,15){\makebox(0,0)[l]{$^{ \cong
}$}} \put(27,15){\makebox(0,0)[r]{$^{ \psi
}$}}  \put(33,15){\makebox(0,0)[l]{$^{ \cong
}$}}
\end{picture}
\]
commutes.

In \cite[Proposition 2.3(c)]{Lyube97}, Lyubeznik proved that an arbitrary $F$-finite $F$-module ${\mathcal M}$ has a root, $L$ say. We provide a short proof in Section \ref{mm} that, for each maximal ideal $\fm$ of $R$, the Bass numbers $\mu^0(\fm, L)$ and $\mu^0(\fm, {\mathcal M})$ are equal. This point is crucial for our algorithm.
\end{ntn}

\section{\bf Preparatory results about modules of generalized fractions}
\label{gf}

%New material follows.
The concept of module of generalized fractions (due to the second author and H. Zakeri \cite{32}) will be used in this paper. The construction and basic properties of these modules can be found in \cite{32}, but, at the request of the referee, we include in this section explanation of some of the main ideas. Throughout this section, we shall work over $A$: see \ref{GLF.1}. Let $n \in \N$ and let $D_n(A)$ denote the set of all $n \times n$ lower triangular matrices with entries in $A$. The determinant of a square matrix $\mathbf{H}$ with entries in $A$ will be denoted by $|\mathbf{H}|$. Of course, the determinant of a lower triangular matrix is the product of its diagonal entries.

\begin{rmd}\label{tri} (Sharp--Zakeri \cite
[\S 2]{32}.)
\begin{enumerate}
\item Let $U$ be a triangular subset of $A^n$, that is, a non-empty subset of $A^n$ such that (a) whenever $(u_1, \ldots, u_n) \in U$ and $\alpha_1, \ldots, \alpha_n \in \N$ (the set of all positive integers), then $(u_1^{\alpha_1}, \ldots, u_n^{\alpha_n}) \in U$ also; and
 (b) whenever $(u_1, \ldots, u_n), (v_1, \ldots, v_n) \in U$, then there exists $(w_1, \ldots, w_n) \in U$ such that $w_i \in \left(\sum_{j=1}^i u_jA\right) \cap \left(\sum_{j=1}^i v_jA\right)$ for all $i=1,\ldots,n$, so that there exist $\mathbf{H}, \mathbf{K} \in D_n(A)$ such that $$\mathbf{H}[u_1, \ldots, u_n]^T = [w_1, \ldots, w_n]^T = \mathbf{K}[v_1, \ldots, v_n]^T.$$  (Here, $^T$ denotes matrix transpose, and $[z_1, \ldots, z_n]^T$ (for $z_1, \ldots, z_n \in A$) is to be interpreted as an $n \times 1$ column matrix in the obvious way.)
\item Let $M$ be an $A$-module. Define a relation $\sim$ on $M \times U$ as follows: for $m, g \in M$ and $(u_1, \ldots, u_n), (v_1, \ldots, v_n) \in U$, write $(m,(u_1, \ldots, u_n)) \sim (g,(v_1, \ldots, v_n))$ precisely when there exist $(w_1, \ldots, w_n) \in U$ and $\mathbf{H}, \mathbf{K} \in D_n(A)$ such that $\mathbf{H}[u_1, \ldots, u_n]^T = [w_1, \ldots, w_n]^T = \mathbf{K}[v_1, \ldots, v_n]^T$ and $|\mathbf{H}|m - |\mathbf{K}|g \in \sum_{j=1}^{n-1} w_jM$.
\item It turns out that $\sim$ is an equivalence relation on $M \times U$. A crucial ingredient for the proof of this given in \cite{32} is what might be called `the Two Routes Lemma' \cite[Lemma 2.3]{32}: suppose that $\mathbf{u}:=(u_1, \ldots, u_n), \mathbf{v}:=(v_1, \ldots, v_n) \in U$ and there exist $\mathbf{H_1},\mathbf{H_2} \in D_n(A)$ such that $\mathbf{H_1}\mathbf{u}^T = \mathbf{v}^T = \mathbf{H_2}\mathbf{u}^T$ (so that, in a sense, there are two `routes' from $\mathbf{u}^T$ to $\mathbf{v}^T$ via lower triangular matrices). Then $|\mathbf{D}\mathbf{H_1}|-|\mathbf{D}\mathbf{H_2}| \in \sum_{i=1}^{n-1} Av_i^2$, where $\mathbf{D}$ is the diagonal matrix $\diag(v_1, \ldots, v_n)$.
\item For $m \in M$ and $\mathbf{u} :=(u_1, \ldots, u_n) \in U$, we denote the equivalence class of $(m,(u_1, \ldots, u_n))$ under $\sim$ by the `generalized fraction'
    $$ \frac{m}{(u_1, \ldots, u_n)}$$ or $m/\mathbf{u}$. The set $U^{-n}M$ of all equivalence classes of $\sim$ is an $A$-module, called the {\em module of generalized fractions of $M$ with respect to $U$,\/} under operations for which, for $m, g \in M$ and $\mathbf{u} :=(u_1, \ldots, u_n), \mathbf{v} := (v_1, \ldots, v_n) \in U$,
\[ \frac{m}{(u_1, \ldots, u_n)} + \frac{g}{(v_1, \ldots, v_n)} = \frac{|\mathbf{H}|m + |\mathbf{K}|g}{(w_1, \ldots, w_n)}
\]
for {\em any\/} choice of $\mathbf{w}:=(w_1, \ldots, w_n) \in U$ and $\mathbf{H},\mathbf{K} \in D_n(A)$ such that $\mathbf{H}\mathbf{u}^T = \mathbf{w}^T = \mathbf{K}\mathbf{v}^T$, and $a(m/\mathbf{u}) = am/\mathbf{u}$ for $a \in A$.
\item The reader should note that the operations in $U^{-n}M$ are such that, if $\mathbf{u} := (u_1, \ldots, u_n) \in U$ and $m \in \sum_{i=1}^{n-1} u_i M$, then the generalized fraction $m/(u_1, \ldots, u_n)$ in $U^{-n}M$ is zero, because $\mathbf{I_n}\mathbf{u}^T = \mathbf{u}^T = \mathbf{I_n}\mathbf{u}^T$ and $|\mathbf{I_n}|m - |\mathbf{I_n}|0 \in \sum_{i=1}^{n-1} u_i M$.
\end{enumerate}
\end{rmd}
%New material: we have included \ref{tri}(v) so that we have been able to replace references to [20. 3.3(ii)] by \ref{tri}(v). We have done this at  the request of the referee. We have also included detail in \ref{tri}(iii) to enable us to replace references to [20. 2.3] by \ref{tri}(iii).End of comment about new material.

The referee has asked us to point out explicitly how a `classical' module of fractions is indeed a module of generalized fractions. We do this in the next lemma.

\begin{lem}\label{mcs} Let $S$ be a multiplicatively closed subset of $A$ and let $M$ be an $A$-module. Then $S$ is a triangular subset of $A^1$, and the module of generalized fractions $S^{-1}M$ is just the classical module of fractions of $M$ with respect to $S$.
\end{lem}

\begin{proof}[Proof.] It is clear that $S$ is a triangular subset of $A^1$, because $[t][s] = [st] = [s][t]$ for all $s,t \in S$.

Consider the relation $\sim$ on $M\times S$ of \ref{tri}(ii). Let $s,t \in S$ and $m,g \in M$, and suppose that $(m,s) \sim (g,t)$. This means that there exist $u\in S$ and $h,k \in A$ such that $hs = u = kt$ and $hm - kg =  0$; multiply by $st$ to see that $sthm - stkg =  0$, that is, $u(tm - sg) = 0$.  Conversely, if there exists $v \in S$ such that $v(tm - sg) = 0$, then $vts \in S$ and $(vt)s = vts = (vs)t$, with $(vt)m - (vs)g = 0$, so that $(m,s) \sim (g,t)$. Thus $\sim$ is just the equivalence relation used to form the classical module of fractions. Therefore, as sets, the module of generalized fractions $S^{-1}M$ is equal to the classical module of fractions, and it is straightforward to check that the arithmetic operations in the two structures are the same.
\end{proof}
%That ends this section of new material.

\begin{ntn}
\label{gf.101} %(Removed material} For terminology and notation concerning
%modules of generalized fractions, the reader is referred to Sharp--Zakeri \cite{32} or to the survey in Huneke--Katzman--Sharp--Yao \cite[\S 2]{80}. %In particular, the concept of {\em triangular subset,\/} which is fundamental to the theory of generalized fractions, is defined in \cite[Definition %2.1]{32} and in \cite[Reminder 2.1]{80}.

%Throughout this section, we shall work over $A$: see \ref{GLF.1}. (End of removed material)

Let $n \in \N$ and let $\mathbf{f} = (f_1, \ldots, f_n) \in A^n$.  We are going to use the module of generalized fractions $A_{\mathbf{f}} = U_{\mathbf{f}}^{-n}A$ where
$U_{\mathbf{f}}$ denotes the triangular subset $\{(f_1^{\alpha_1}, \ldots, f_n^{\alpha_n}) : \alpha_1, \ldots, \alpha_n \in \nn\}$ of $A^n$.
A general element $\Phi$ of $A_{\mathbf{f}}$ has the form $a/(f_1^{\beta_1}, \ldots, f_n^{\beta_n})$ for some $a \in A$ and $\beta_1, \ldots, \beta_n \in \nn$. Use of the diagonal matrix $\diag(f_1^{\beta - \beta_1}, \ldots, f_n^{\beta - \beta_n})$ for a $\beta \in \N$ greater than all the $\beta_i$ enables us to see that
\[
\Phi = \frac{a}{(f_1^{\beta_1}, \ldots, f_n^{\beta_n})} = \frac{f_1^{\beta - \beta_1} \ldots f_n^{\beta - \beta_n}a}{(f_1^{\beta}, \ldots, f_n^{\beta})}.
\]
So, when considering a general element $\Phi$ of $A_{\mathbf{f}}$ as above, we may assume that $\beta_1 = \cdots = \beta_n$. Thus
$$
A_{\mathbf{f}} = \bigcup_{\beta \in \N} A \frac{1}{(f_1^{\beta}, \ldots, f_n^{\beta})}
$$
is the union of the cyclic submodules $A \big(1/(f_1^{\beta}, \ldots, f_n^{\beta})\big)~(\beta \in \N)$.  To work with these, we would like to have descriptions of their annihilators.
\end{ntn}

\begin{lem}
\label{gf.102} Let the notation be as in\/ {\rm \ref{gf.101}}, and let $\beta \in \N$. Then the annihilator of the generalized fraction $1/(f_1^{\beta}, \ldots, f_n^{\beta}) \in A_{\mathbf{f}}$ is
\[
\bigcup_{j\in\nn} (f_1^{j+\beta}A + \cdots + f_{n-1}^{j+\beta}A : f_1^{j} \ldots f_n^{j}).
\]
\end{lem}

\begin{proof} Let $j \in \nn$ and $a \in (f_1^{j+\beta}A + \cdots + f_{n-1}^{j+\beta}A : f_1^{j} \ldots f_n^{j})$.  Use $\diag(f_1^{j}, \ldots, f_n^{j})$ to see that
\[
a \frac{1}{(f_1^{\beta}, \ldots, f_n^{\beta})} = \frac{f_1^{j} \ldots f_n^{j}a}{(f_1^{j+\beta}, \ldots, f_n^{j+\beta})},
\]
and this is zero by \ref{tri}(v).

Now let $a \in A$ be such that $a (1/(f_1^{\beta}, \ldots, f_n^{\beta})) = 0$. This means that there exist an $\mathbf{H} \in D_n(A)$ and $\alpha_1, \ldots, \alpha_n \in \nn$ such that $\mathbf{H} \big[ f_1^{\beta} \ldots f_n^{\beta}\big]^T = \big[ f_1^{\alpha_1} \ldots f_n^{\alpha_n}\big]^T$ and $|\mathbf{H}|a \in \sum_{j=1}^{n-1} f_j^{\alpha_j}A$. Let $\delta = \max\{\beta,\alpha_1, \ldots, \alpha_n\}$. Set
$$
\mathbf{D}_1 := \diag (f_1^{\delta - \beta}, \ldots, f_n^{\delta - \beta}), \quad \mathbf{D}_2 := \diag (f_1^{\delta - \alpha_1}, \ldots, f_n^{\delta - \alpha_n}).
$$
Then $\mathbf{D}_2\mathbf{H} \big[ f_1^{\beta} \ldots f_n^{\beta}\big]^T = \big[ f_1^{\delta} \ldots f_n^{\delta}\big]^T = \mathbf{D}_1 \big[ f_1^{\beta} \ldots f_n^{\beta}\big]^T$. Let $\mathbf{E} := \diag (f_1^{\delta}, \ldots, f_n^{\delta})$.  By the Two Routes Lemma (see \ref{tri}(iii)),
$$
|\mathbf{E}\mathbf{D}_2\mathbf{H}| - |\mathbf{E}\mathbf{D}_1| \in {\textstyle \sum_{j=1}^{n-1} f_j^{2\delta}A.}
$$
Since $|\mathbf{H}|a \in \sum_{j=1}^{n-1} f_j^{\alpha_j}A$, it follows that $|\mathbf{E}\mathbf{D}_2\mathbf{H}|a \in \sum_{j=1}^{n-1} f_j^{2\delta}A$.  Therefore $|\mathbf{E}\mathbf{D}_1|a \in \sum_{j=1}^{n-1} f_j^{2\delta}A$, that is,
$
a \in \big({\textstyle \sum_{j=1}^{n-1} f_j^{2\delta}A: f_1^{2\delta - \beta}\ldots f_n^{2\delta - \beta}}\big).
$
\end{proof}

\begin{defi}
\label{gf.103} Suppose, in the situation of \ref{gf.102}, that $A$ is Noetherian. The ideals in the sequence $$\big((f_1^{j+1}A + \cdots +f_{n-1}^{j+1}A: f_1^j \ldots f_n^j)\big)_{j=1}^\infty$$ form an ascending chain which will eventually become stationary; we call the eventual stationary value the {\em lower limit ideal of $(f_1, \ldots, f_n)$\/} and denote it by $(f_1, \ldots, f_n)^{\lowlim}$. Observe that $$f_{n+1}(f_1, \ldots, f_n)^{\lowlim} \subseteq (f_1, \ldots, f_n,f_{n+1})^{\lowlim}.$$

This contrasts with the {\em limit closure of $(f_1, \ldots, f_n)$}, denoted by $(f_1, \ldots, f_n)^{\lim}$, and defined by Huneke (in a special case) \cite[Definition 5.3]{Hunek98} as
\[
(f_1, \ldots, f_n)^{\lim}:= \bigcup_{j=1}^{\infty}\big((f_1^{j+1}A + \cdots +f_{n}^{j+1}A: f_1^j \ldots f_n^j)\big).
\]
It follows from \ref{gf.102} that the annihilator of the generalized fraction $1/(f_1, \ldots, f_n)$ in $A_{\mathbf f}$ is
$$
\left(f_1, \ldots, f_n\right)^{\lowlim} = (f_1^{j+1}A + \cdots +f_{n-1}^{j+1}A: f_1^j \ldots f_n^j)  \quad \mbox{for all $j \gg 0$}.
$$
The same result shows that $(f_1, \ldots, f_n)^{\lim}$ is the annihilator of the generalized fraction $1/(f_1, \ldots, f_n,1)$ in $(U_{\mathbf f} \times \{1\})^{-(n+1)}A$.

Let $t \in \N$. The annihilator of the generalized fraction $1/(f_1^{p^t}, \ldots, f_n^{p^t})$ in $A_{\mathbf f}$ is, by Lemma \ref{gf.102},
\[
\bigcup_{j\in\nn} (f_1^{j+p^t}\!A + \cdots + f_{n-1}^{j+p^t}\!A : f_1^{j} \ldots f_n^{j}).
\]
Since the ideals in the sequence $\big((f_1^{j+p^t}\!A + \cdots + f_{n-1}^{j+p^t}\!A : f_1^{j} \ldots f_n^{j})\big)_{j\in \nn}$ form an ascending chain, this annihilator is equal to
\[
\bigcup_{k\in\nn} \big(f_1^{kp^t+p^t}\!A + \cdots + f_{n-1}^{kp^t+p^t}\!A : f_1^{kp^t} \ldots f_n^{kp^t}\big) = (f_1^{p^t}, \ldots, f_n^{p^t})^{\lowlim}.
\]
\end{defi}

Recall that the $p$th Frobenius power of $\fa$, denoted $\fa^{[p]}$, is the ideal generated by all $p$th powers of elements of $\fa$.

\begin{lem}
\label{gf.104} Let $n \in \N$ and let ${\mathbf f} = (f_1, \ldots, f_n) \in R^n$ (see {\rm \ref{GLF.1}\/}). Then
\[
\left(f_1^p, \ldots, f_n^p\right)^{\lowlim}= \big(\left(f_1, \ldots, f_n\right)^{\lowlim}\big)^{[p]}.
\]
\end{lem}

\begin{proof} Choose $j$ sufficiently large so that $\left(f_1, \ldots, f_n\right)^{\lowlim} = (f_1^{j+1}R + \cdots +f_{n-1}^{j+1}R: f_1^j \ldots f_n^j)$ and $\left(f_1^p, \ldots, f_n^p\right)^{\lowlim} = (f_1^{p(j+1)}R + \cdots +f_{n-1}^{p(j+1)}R: f_1^{pj} \ldots f_n^{pj})$. Since $f$ is a flat ring homomorphism,
\begin{align*}
\left(f_1^p, \ldots, f_n^p\right)^{\lowlim} & = (f_1^{p(j+1)}R + \cdots +f_{n-1}^{p(j+1)}R: f_1^{pj}\ldots f_n^{pj})\\
&= (f_1^{j+1}R + \cdots +f_{n-1}^{j+1}R: f_1^j \ldots f_n^j)^{[p]}
= \big(\left(f_1, \ldots, f_n\right)^{\lowlim}\big)^{[p]}.    \qedhere
\end{align*}
\end{proof}

\begin{lem}
\label{gf.105} Let $n \in \N$ and let ${\mathbf f} = (f_1, \ldots, f_n) \in R^n$. For each $e \in \N$, let
\[
\fg_e := (f_1^{p^e}R + \cdots +f_{n}^{p^e}R: f_1^{p^e-1} \ldots f_n^{p^e-1}).
\]
If $\fg_e = \fg_{e+1}$, then $\fg_{e'} = \fg_{e}$ for all $e' \geq e$.
\end{lem}

\begin{note} Since the ideals in the sequence $$\big((f_1^{j}R + \cdots +f_{n}^{j}R: f_1^{j-1} \ldots f_n^{j-1})\big)_{j=1}^\infty$$ form an ascending chain, the above lemma provides us with an effective way to compute $\left( f_1, \ldots, f_n\right)^{\lim}$.
\end{note}

\begin{proof} Let $\fc_e = \sum_{i=1}^n f_i^{p^e}R$, for each $e \in \nn$; let $r$ denote a general element of $R$.  When $e \geq 1$, there is an isomorphism $\beta_e : F(R/\fc_{e-1}) \stackrel{\cong}{\lra} R/\fc_e$ which maps $r \otimes (1 + \fc_{e-1}))$ to $r + \fc_e$. Denote the isomorphism
$$
\beta_e \circ F(\beta_{e-1}) \circ \cdots \circ F^{e-1}(\beta_1) : F^e(R/\fc_0) \stackrel{\cong}{\lra}R/\fc_e
$$
by $\mu_e$; it maps $r \otimes (1 \otimes \cdots \otimes (1 + \fc_0) \cdots )$ to $r + \fc_e$.

Set $g = f_1 \ldots f_n$.
Let $\gamma : R/\fc_0 \lra R/\fc_1$ be the $R$-homomorphism induced by multiplication by $g^{p-1}$. Let $\lambda := \beta_1^{-1} \circ \gamma : R/\fc_0 \lra F(R/\fc_0)$. Note that $\lambda(r+\fc_0) = g^{p-1}r \otimes (1 + \fc_0)$. Set
$$
\delta_e := F^{e-1}(\lambda)\circ F^{e-2}(\lambda)\circ \cdots \circ F(\lambda)\circ \lambda : R/\fc_0 \lra F^e(R/\fc_0),
$$
which maps $r+\fc_0$ to
$$
g^{p^e-p^{e-1}}g^{p^{e-1}-p^{e-2}}\ldots g^{p-1}r \otimes (1 \otimes \cdots \otimes (1 + \fc_0) \cdots ) = g^{p^e-1}r \otimes (1 \otimes \cdots \otimes (1 + \fc_0) \cdots ).
$$
By Lyubeznik \cite[Proposition 2.3(b)]{Lyube97}, not only does the chain $\Ker \delta_1 \subseteq \Ker \delta_2 \subseteq \cdots \subseteq \Ker \delta_k \subseteq \cdots$ become stationary, but also, if $\Ker \delta_e = \Ker \delta_{e+1}$ for some $e \in \N$, then $\Ker \delta_{e'} = \Ker \delta_e$ for all $e' \geq e$. Now $\Ker \delta_e = \Ker (\mu_e \circ \delta_e)$ and $\mu_e \circ \delta_e : R/\fc_0 \lra R/\fc_e$ is  induced by multiplication by $g^{p^e-1}$. Therefore $\Ker \delta_e = \fg_e/\fc_0$.  The result follows.
\end{proof}

\begin{lem}
\label{gf.106} With the notation of\/ {\rm \ref{gf.105}} and with $n > 1$,
\[
\left(f_1, \ldots, f_n\right)^{\lowlim} = \big(\left(f_1, \ldots, f_{n-1}\right)^{\lim} : f_n^j\big)
\]
for any $j \in \nn$ such that $\big(\left(f_1, \ldots, f_{n-1}\right)^{\lim} : f_n^j\big) = \big(\left(f_1, \ldots, f_{n-1}\right)^{\lim} : f_n^{j+1}\big)$ (and there will be one such because $R$ is Noetherian).
\end{lem}

\begin{note} It follows from the note immediately after the statement of Lemma \ref{gf.105} that that lemma provides us with an effective way to compute $\left(f_1, \ldots, f_{n-1}\right)^{\lim}$, and so Lemma \ref{gf.106} provides us with an effective way to compute $\left(f_1, \ldots, f_n\right)^{\lowlim}$.
\end{note}

\begin{proof} Observe that, for any ideal $\fd$ of $R$ and $j \in \nn$, if $(\fd : f_n^j) = (\fd : f_n^{j+1})$, then $(\fd : f_n^{j+1}) = (\fd : f_n^{j+2})$. Let $h$ be the smallest non-negative integer such that $$\big(\left(f_1, \ldots, f_{n-1}\right)^{\lim} : f_n^h\big) = \big(\left(f_1, \ldots, f_{n-1}\right)^{\lim} : f_n^{h+1}\big).$$

Because the definitions of lower limit and limit closure involve unions of ascending chains of ideals, there exists an integer $t \in \N$ such that
$$\left(f_1, \ldots, f_n\right)^{\lowlim} = \big(f_1^{j+1}R + \cdots +f_{n-1}^{j+1}R: f_1^{j}\ldots f_n^{j}\big) \quad \mbox{for all~} j \geq t $$ and
$$ \left(f_1, \ldots, f_{n-1}\right)^{\lim} = \big(f_1^{j+1}R + \cdots +f_{n-1}^{j+1}R: f_1^{j}\ldots f_{n-1}^{j}\big) \quad \mbox{for all~} j \geq t. $$
Let $w \in \N$ be such that $w \geq \max\{h,t\}$.  Then we have
\begin{align*}
\left(f_1, \ldots, f_n\right)^{\lowlim} & = \big(f_1^{w+1}R + \cdots +f_{n-1}^{w+1}R: f_1^{w}\ldots f_n^{w}\big) \\ & =
\big(\left(f_1^{w+1}R + \cdots +f_{n-1}^{w+1}R: f_1^{w}\ldots f_{n-1}^{w}\right):f_n^w\big)
\\ &= \left(\left(\left(f_1, \ldots, f_{n-1}\right)^{\lim}\right):f_n^w\right)\\ &= \big(\big(\left(f_1, \ldots, f_{n-1}\right)^{\lim}\big):f_n^j\big)  \quad \mbox{for all~} j \geq h.
\end{align*}
The claim follows from this.
\end{proof}

We shall use the following technical lemma about generalized fractions in the next section.

\begin{lem}
\label{calc.6} Let $M$ be a module over the commutative ring $A$, and let $U$ be a triangular subset of $A^n$. There is an $A$-isomorphism $\mu_M : U^{-n}A \otimes_AM \stackrel{\cong}{\lra} U^{-n}M$ for which $\mu_M(a/\mathbf{u} \otimes x) = ax/\mathbf{u}$ for all $a \in A$, $x \in M$ and $\mathbf{u} \in U$.
\end{lem}

\begin{proof}[Proof.] It is straightforward to show that there is a map $\lambda : U^{-n}A \times M \lra U^{-n}M$ for which $\lambda((a/\mathbf{u},x)) = ax/\mathbf{u}$ for all $a \in A$, $x \in M$ and $\mathbf{u} \in U$, although one must remember that a generalized fraction in $U^{-n}A$ can be represented in many different ways as  $a/(u_1, \ldots, u_n)$.

Since two generalized fractions in $U^{-n}A$ can be put on a common denominator, it is then clear that $\lambda$ is $A$-linear in both variables and gives rise to an $A$-homomorphism $\mu_M : U^{-n}A \otimes_AM \lra U^{-n}M$ satisfying the formula in the statement.

We construct an inverse for $\mu_M$. Suppose that a generalized fraction in $U^{-n}M$ is represented in two ways as  $x/ (u_1, \ldots, u_n)$ and $x'/ (u'_1, \ldots, u'_n)$ for $x,x' \in M$ and $\mathbf{u} = (u_1, \ldots, u_n), \mathbf{u'} = (u'_1, \ldots, u'_n) \in U$. Then there exist $\mathbf{u^{\prime\prime}} := (u_1^{\prime\prime}, \ldots, u_n^{\prime\prime}) \in U$ and $\mathbf{J}, \mathbf{J'} \in D_n(A)$ such that
\[
\mathbf{J}\mathbf{u}^T = \mathbf{u^{\prime\prime}}^T =\mathbf{J'}\mathbf{u'}^T \quad \mbox{and} \quad |\mathbf{J}|x - |\mathbf{J'}|x' = {\textstyle \sum_{i=1}^{n-1} u_i^{\prime\prime}x_i}
\]
for some $x_1, \ldots, x_{n-1} \in M$. Therefore, in $U^{-n}A\otimes_A M $, we have
\begin{align*}
\frac{1}{\mathbf{u}} \otimes x - \frac{1}{\mathbf{u'}} \otimes x' & = \frac{|\mathbf{J}|}{\mathbf{u^{\prime\prime}}} \otimes x - \frac{|\mathbf{J'}|}{\mathbf{u^{\prime\prime}}} \otimes x'=\frac{1}{\mathbf{u^{\prime\prime}}}\otimes \left(|\mathbf{J}|x - |\mathbf{J'}|x'\right) \\ & = \sum_{i=1}^{n-1} \frac{1}{(u_1^{\prime\prime}, \ldots, u_n^{\prime\prime})} \otimes u_i^{\prime\prime}x_i  = \sum_{i=1}^{n-1}  \frac{u_i^{\prime\prime}}{(u_1^{\prime\prime}, \ldots, u_n^{\prime\prime})}\otimes x_i = 0 \quad \mbox{by {\rm \ref{tri}(v)}}.
\end{align*}
It follows that there is a mapping $\nu_M : U^{-n}M \lra M \otimes_AU^{-n}A $ for which $\nu_M(x/\mathbf{u}) = x \otimes 1/\mathbf{u}$ for all $x \in M$ and $\mathbf{u} \in U$, and one checks easily that $\nu_M$ is an inverse for $\mu_M$.
\end{proof}

%The title of \S 3 has been changed. \section{\bf Certain modules of generalized fractions are $F$-modules}
\section{\bf Modules of generalized fractions of $F$-modules are again $F$-modules.}
\label{gfF}

\begin{ntn}
\label{gfF.0} Throughout this section, we shall work over $R$: see \ref{GLF.1}. Also, $U$ will denote a triangular subset of $R^n$.
\end{ntn}

The main aim of this section is to show that, whenever $\mathcal{M}$ is an $F$-module over $R$, then the module of generalized fractions $U^{-n}\mathcal{M}$ is again an $F$-module. We shall first achieve this result for the special case in which ${\mathcal M} = R$, and then we shall use \ref{calc.6} to prove the general result.

\begin{prop}
\label{gf.106xy} The module of generalized fractions $U^{-n}R$ is an $F$-module with structural isomorphism $\theta : U^{-n}R \stackrel{\cong}{\lra} F(U^{-n}R) = R' \otimes_RU^{-n}R$ such that
\[
\theta\left( \frac{r}{(u_1,\ldots,u_n)}\right) = u_1^{p-1}\ldots u_n^{p-1}r\otimes\frac{1}{(u_1,\ldots,u_n)} \quad \mbox{for all $(u_1,\ldots,u_n) \in U$ and $r\in R$}.
\]
\end{prop}

\begin{proof} Suppose that $r,s \in R$ and $\mathbf{u}=(u_1,\ldots,u_n), \mathbf{v}=(v_1,\ldots,v_n) \in U$ are such that $r/\mathbf{u} = s/ \mathbf{v}$ in $U^{-n}R$.
Then there exist $\mathbf{w} = (w_1, \ldots, w_n) \in U$ and $\mathbf{H}, \mathbf{K} \in D_n(R)$ such that
\[
\mathbf{H}\mathbf{u}^T = \mathbf{w}^T = \mathbf{K}\mathbf{v}^T \quad \mbox{and} \quad |\mathbf{H}|r - |\mathbf{K}|s = w_1a_1 + \cdots + w_{n-1}a_{n-1} \quad \mbox{for some $a_1, \ldots, a_{n-1} \in R$.}
\]
Let $\mathbf{D}_1 = \diag(u_1, \ldots, u_n)$, $\mathbf{D}_2 = \diag(v_1, \ldots, v_n)$  and $\mathbf{D}_3 = \diag(w_1, \ldots, w_n)$. Our immediate aim is to show that
\[
u_1^{p-1}\ldots u_n^{p-1}r \otimes \frac{1}{(u_1, \ldots, u_n)} = v_1^{p-1}\ldots v_n^{p-1}s \otimes \frac{1}{(v_1, \ldots, v_n)} \quad \mbox{in $R'\otimes_RU^{-n}R$,}
\]
and we shall achieve this by showing that
\begin{equation}\label{A}
\Theta := \left(|\mathbf{H}|^p|\mathbf{D}_1|^{p-1}r  - |\mathbf{K}|^p|\mathbf{D}_2|^{p-1}s\right) \otimes \frac{1}{(w_1, \ldots, w_n)} \quad \mbox{in $R'\otimes_RU^{-n}R$}
\end{equation}
is zero. Set $\Phi := 1/(w_1^2, \ldots, w_n^2) \in U^{-n}R$. Note that
\begin{equation}\label{B}
\Theta = |\mathbf{D}_3|^{p}\left(|\mathbf{H}|^p|\mathbf{D}_1|^{p-1}r - |\mathbf{K}|^p|\mathbf{D}_2|^{p-1}s \right)\otimes \Phi.
\end{equation}
Note also that
\[
\mathbf{H}^p\mathbf{D}_1^{p-1}\mathbf{u}^T = \mathbf{w}^{pT} = \mathbf{D}_3^{p-1}\mathbf{H}\mathbf{u}^T \quad \mbox{and} \quad \mathbf{K}^p\mathbf{D}_2^{p-1}\mathbf{v}^T = \mathbf{w}^{pT} = \mathbf{D}_3^{p-1}\mathbf{K}\mathbf{v}^T.
\]
Therefore, by \ref{tri}(iii), we have
$ |\mathbf{D}_3|^p|\mathbf{H}|^p|\mathbf{D}_1|^{p-1} - |\mathbf{D}_3|^p|\mathbf{D}_3|^{p-1}|\mathbf{H}| = w_1^{2p}b_1 + \cdots + w_{n-1}^{2p}b_{n-1}$ and $|\mathbf{D}_3|^p|\mathbf{K}|^p|\mathbf{D}_2|^{p-1} - |\mathbf{D}_3|^p|\mathbf{D}_3|^{p-1}|\mathbf{K}|  = w_1^{2p}c_1 + \cdots + w_{n-1}^{2p}c_{n-1}$ for some $b_1, \ldots, b_{n-1},c_1, \ldots, c_{n-1} \in R$. Therefore
\[
\left(|\mathbf{D}_3|^p|\mathbf{H}|^p|\mathbf{D}_1|^{p-1}r - |\mathbf{D}_3|^p|\mathbf{D}_3|^{p-1}|\mathbf{H}|r \right) \otimes \Phi = \left(\sum_{j=1}^{n-1}w_j^{2p}b_jr\right)\otimes \Phi = \sum_{j=1}^{n-1}b_jr \otimes \frac{w_j^2}{(w_1^2, \ldots, w_n^2)}= 0
\]
by \ref{tri}(v). Similarly,
$
|\mathbf{D}_3|^{p}|\mathbf{K}|^p|\mathbf{D}_2|^{p-1}s\otimes \Phi = |\mathbf{D}_3|^p|\mathbf{D}_3|^{p-1}|\mathbf{K}|s\otimes \Phi.
$
It follows from Equation (2) that
\begin{align*}
\Theta &= |\mathbf{D}_3|^{p}\left(|\mathbf{D}_3|^{p-1}|\mathbf{H}|r - |\mathbf{D}_3|^{p-1}|\mathbf{K}|s \right)\otimes \Phi =
|\mathbf{D}_3|^{2p-1}\left(|\mathbf{H}|r - |\mathbf{K}|s\right)\otimes \Phi\\&=
w_1^{2p-1} \ldots w_n^{2p-1}\left( {\textstyle \sum_{j=1}^{n-1}w_ja_j}\right)\otimes \Phi =
{\textstyle \sum_{j=1}^{n-1}w_j^{2p}y_j\otimes\Phi }\quad \mbox{for some $y_1, \ldots, y_{n-1} \in R'$} \\ &=  \sum_{j=1}^{n-1}y_j \otimes \frac{w_j^2}{(w_1^2, \ldots, w_n^2)} = 0 \quad \mbox{by {\rm \ref{tri}(v)}}.
\end{align*}
There is therefore a mapping $\theta : U^{-n}R \lra F(U^{-n}R) = R' \otimes_RU^{-n}R$ such that
\[
\theta\left( \frac{r}{(u_1,\ldots,u_n)}\right) = u_1^{p-1}\ldots u_n^{p-1}r\otimes\frac{1}{(u_1,\ldots,u_n)} \quad \mbox{for all $(u_1,\ldots,u_n) \in U$ and $r\in R$}.
\]

By \cite[Lemma 3.5]{KS}, there is an $R$-homomorphism $\phi : R'\otimes_RU^{-n}R  \lra U^{-n}R$ for which
\[
\phi \left(r' \otimes \frac{r}{(u_1,\ldots,u_n)}\right) = r'\frac{r^p}{(u_1^p,\ldots,u_n^p)}
\]
for all $r' \in R$, $r \in R$ and $(u_1,\ldots,u_n) \in U$. One checks easily that $\theta$ and $\phi$ are inverse isomorphisms.
\end{proof}

We have not so far been able to find the following lemma in the existing literature.

%We have replaced $\Delta$ with $\Omega$ throughout \S3 from the beginning of \ref{add.1} and throughout \S4.

\begin{lem}
\label{add.1} Let $X$ and $Y$ be $R$-modules.  There is an $R$-isomorphism
$$
\Delta : F(X) \otimes_R F(Y) \stackrel{\cong}{\lra} F(X \otimes_R Y)
$$
for which $\Delta((r' \otimes x) \otimes (s' \otimes y)) = r's' \otimes (x \otimes y)$ for all $r',s' \in R'$ and $x \in X$, $y \in Y$.
\end{lem}

\begin{note} The reader may note that the argument in the proof below is valid over any commutative Noetherian ring of characteristic $p$; the hypothesis that the ring is regular is not needed here.
\end{note}

\begin{proof}[Proof.] Let $r',s'$ denote general elements of $R'$, let $a,b$ denote general elements of $R$, and let $x$ (respectively $y$) denote a general element of $X$ (respectively $Y$).

In the formation of $F(X) \otimes_R F(Y)$, the left $R$-module $F(Y) = R' \otimes_RY$ is such that $b(s' \otimes y) = bs' \otimes y$ and the right $R$-module $W := F(X)$ is such that $(r'\otimes x)a = ar' \otimes x$. Let $w$ denote a general element of $W$. By the associative law for tensor products, there is a $\Z$-isomorphism
$$
\Delta_1 : F(X) \otimes_R F(Y)= W \otimes_R (R'\otimes_RY) \stackrel{\cong}{\lra} (W \otimes_R R')\otimes_RY
$$
for which $\Delta_1(w \otimes(s' \otimes y)) = (w \otimes s') \otimes y$. Note that, in the formation of $(W \otimes_R R')\otimes_RY$, the right $R$-module structure on $W \otimes_R R'$ is such that $(w \otimes s')a = w \otimes s'a^p$.

Since the left $R$-module structure on $R'= R$ is the natural one, there is a $\Z$-isomorphism $\Gamma : W \otimes_R R' \stackrel{\cong}{\lra} W$ such that $\Gamma (w \otimes s') = s'w$. Let $\widetilde{W}$ denote the Abelian group $W$ endowed with the right $R$-module structure that makes $\Gamma$ into an isomorphism of right $R$-modules. This is such that $\widetilde{w}b = \Gamma ( \Gamma^{-1}(\widetilde{w})b)$ for all $\widetilde{w} \in \widetilde{W}$, that is, $$(r'\otimes x)b = \Gamma(((r'\otimes x)\otimes 1)b) = \Gamma((r'\otimes x)\otimes b^p) = b^p(r'\otimes x) = b^pr'\otimes x = r'\otimes bx.$$ Thus this right $R$-module structure on $\widetilde{W}$ is the structure induced on $R' \otimes_R X$ by regarding $X$ as both a right $R$-module and a left $R$-module in the natural way. The reader should note the difference between $\widetilde{W}$ and the $R$-module $W$ considered in the second paragraph of this proof.

The isomorphism of right $R$-modules $\Gamma$ induces a $\Z$-isomorphism
$$
\Delta_2 := \Gamma \otimes \Id_Y : (W \otimes_R R')\otimes_RY \stackrel{\cong}{\lra} \widetilde{W} \otimes_RY = (R' \otimes_RX)\otimes_RY
$$
which is such that $\Gamma\otimes \Id_Y(((r'\otimes x)\otimes s')\otimes y) = (s'r'\otimes x)\otimes y$.  It is important to note that, in this display, the right $R$-module structure on the right-most appearance of $R' \otimes_RX$ is the one that comes from regarding $X$ as an $(R,R)$-bimodule in the natural way.

Another use of the associative law for tensor products produces a $\Z$-isomorphism $$\Delta_3 : \widetilde{W} \otimes_RY = (R' \otimes_RX)\otimes_RY \stackrel{\cong}{\lra} R' \otimes_R(X\otimes_RY) = F(X\otimes_RY)$$
such that $\Delta_3((r'\otimes x)\otimes y) = r'\otimes (x\otimes y)$. The composition
$$
\Delta := \Delta_3 \circ \Delta_2 \circ \Delta_1 : F(X) \otimes_R F(Y) \stackrel{\cong}{\lra} F(X \otimes_R Y)
$$
satisfies $\Delta((r' \otimes x) \otimes (s' \otimes y)) = r's' \otimes (x \otimes y)$ and is an $R$-isomorphism.
\end{proof}

\begin{thm}
\label{mm.1g}
Let $\mathcal{M}$ be an $F$-module over the regular ring $R$ with structural isomorphism $\nu : \mathcal{M} \stackrel{\cong}{\lra} F(\mathcal{M}) = R'\otimes_R\mathcal{M}$. Then the module of generalized fractions $U^{-n}\mathcal{M}$ is again an $F$-module with structural isomorphism $\kappa : U^{-n}\mathcal{M} \stackrel{\cong}{\lra} F(U^{-n}\mathcal{M}) = R' \otimes_RU^{-n}\mathcal{M}$ such that
\[
\kappa\left( \frac{\nu^{-1}\left(\sum_{i=1}^h r_i' \otimes m_i\right)}{(u_1,\ldots,u_n)} \right) = \sum_{i=1}^h u_1^{p-1}\ldots u_n^{p-1}r_i' \otimes \frac{m_i}{(u_1,\ldots,u_n)}
\]
for all $(u_1,\ldots,u_n) \in U$, $h \in \N$, $r_1', \ldots, r'_h \in R$ and $m_1, \ldots, m_h \in \mathcal{M}$.
\end{thm}

\begin{proof} Define $\kappa$ to be the composition of the isomorphism $\psi : U^{-n}\mathcal{M} \stackrel{\cong}{\lra}U^{-n}R\otimes_R {\mathcal M}$ from \ref{calc.6}, the tensor product $\theta \otimes \nu : U^{-n}R\otimes_R {\mathcal M}\stackrel{\cong}{\lra}F(U^{-n}R)\otimes_R F({\mathcal M})$ where $\theta : U^{-n}R \stackrel{\cong}{\lra} F(U^{-n}R)$ is the isomorphism of \ref{gf.106xy}, the isomorphism $\Delta:F(U^{-n}R)\otimes_R F({\mathcal M})\stackrel{\cong}{\lra} F(U^{-n}R\otimes_R{\mathcal M})$ of \ref{add.1}, and the isomorphism $F(\psi^{-1}): F(U^{-n}R\otimes_R{\mathcal M})\stackrel{\cong}{\lra}
F(U^{-n}{\mathcal M})$, where $\psi$ is as immediately above.
\end{proof}

\begin{rmks}\label{1006} Let $\mathcal{M}$ be an $F$-module over the regular ring $R$. Let ${\mathbf f} = (f_1, \ldots,f_n) \in R^n$.
\begin{enumerate}
\item By \ref{mm.1g}, $\mathcal{M}_{{\mathbf f}} = U_{{\mathbf f}}^{-n}\mathcal{M}$ is again an $F$-module.
\item In particular, $R_{{\mathbf f}} = U_{{\mathbf f}}^{-n}R$ is an $F$-module.
\item In particular again, since a multiplicatively closed subset $S$ of $R$ is a triangular subset of $R^1$, the ordinary ring of fractions $S^{-1}R$ is an $F$-module.
\end{enumerate}
\end{rmks}

We shall find the following useful in Section \ref{mm}.

\begin{rmk}\label{alg.13} Let $S$ be a multiplicatively closed subset of our regular ring $R$, and let $F'$ denote the Frobenius functor over the regular ring $S^{-1}R$ (also of characteristic $p$). Let $M$ be an $R$-module.  Then it is straightforward to show that there is an $S^{-1}R$-isomorphism
\[
\tau_M : S^{-1}(F(M)) = S^{-1}(R'\otimes_RM) \stackrel{\cong}{\lra} (S^{-1}R)' \otimes_{S^{-1}R}S^{-1}M = F'(S^{-1}M)
\]
for which $\tau_M((r'\otimes m)/s) = (r'/s)\otimes (m/1)$ for all $r' \in R'$, $s \in S$ and $m\in M$. (Here, $(S^{-1}R)'$ denotes the ring
$S^{-1}R$ considered as a left module over itself in the natural way and as a right $S^{-1}R$-module via the Frobenius homomorphism.) As $M$ varies through the category of $R$-modules, the $\tau_M$ constitute a natural equivalence of functors.
\end{rmk}

\section{\bf Some modules of generalized fractions are $F$-finite $F$-modules}
\label{mm}

Throughout this section, we shall work over $R$: see \ref{GLF.1}.

In \ref{1006}(ii), we noted that the module of generalized fractions $R_{{\mathbf f}}$, where ${\mathbf f} = (f_1, \ldots,f_n) \in R^n$, is an $F$-module.  One of the aims of this section is to prove that this $F$-module $R_{{\mathbf f}}$ is $F$-finite. In fact, we shall prove that, if ${\mathcal M}$ is any $F$-finite $F$-module, then the module of generalized fractions ${\mathcal M}_{{\mathbf f}}$ is $F$-finite.

\begin{thm}
\label{gf.107} Let $n \in \N$ and let ${\mathbf f} = (f_1, \ldots, f_n) \in R^n$. Let $j \in \nn$. We consider the cyclic submodules of $R_{\mathbf{f}}$ generated by $1/\mathbf{f}^{p^{j+1}}$ and $1/\mathbf{f}^{p^{j}}$.
\begin{enumerate}
\item There is an isomorphism $\psi_{j+1} : R\big(1/\mathbf{f}^{p^{j+1}}\big) \stackrel{\cong}{\lra} F\big(R(1/\mathbf{f}^{p^{j}})\big)$ for which $$\psi_{j+1}(r/\mathbf{f}^{p^{j+1}}) = r \otimes 1/\mathbf{f}^{p^{j}} \quad \mbox{for all $r \in R$}.$$
\item There is an isomorphism $\gamma_{j} : R\big(1/\mathbf{f}^{p^{j}}\big) \stackrel{\cong}{\lra} F^j\big(R(1/\mathbf{f})\big)$ for which $$\gamma_{j}\left(\frac{r}{\mathbf{f}^{p^{j}}}\right) = r \otimes\Big( 1 \otimes \Big( \cdots \otimes \Big(1 \otimes \frac{1}{\mathbf{f}}\Big)\cdots\Big)\Big) \quad \mbox{for all $r \in R$}.$$
(Interpret $\gamma_0$ as the identity mapping on $R\big(1/\mathbf{f}\big)$. Then we can take
$\gamma_{j+1} := F(\gamma_j) \circ \psi_{j+1}$ for all $j \in \nn$.)
\item Let $\widetilde{\theta} : R\frac{1}{\mathbf{f}} \lra F \left(R\frac{1}{\mathbf{f}}\right)$ be the composition of the inclusion map $ R(1/\mathbf{f}) \stackrel{\subseteq}{\lra} R(1/\mathbf{f}^p)$ and the isomorphism $\psi_1 : R(1/\mathbf{f}^p) \stackrel{\cong}{\lra} F \left( R(1/\mathbf{f})\right)$ of part\/ {\rm (i)}.  Then $\widetilde{\theta} (r/\mathbf{f}) = f_1^{p-1}\ldots f_n^{p-1}r \otimes 1/\mathbf{f}$ for all $r \in R$, and $\widetilde{\theta}$ is monomorphic. There is a commutative diagram
\[
\begin{picture}(300,75)(-150,-25)
\put(-120,40){\makebox(0,0){$                   R\frac{1}{\mathbf{f}}
$}}
\put(-60,40){\makebox(0,0){$                    R\frac{1}{\mathbf{f}^p}
$}}
\put(-7,40){\makebox(0,0){$                      \cdots
$}}
\put(52,40){\makebox(0,0){$                     R\frac{1}{\mathbf{f}^{p^j}}
$}}
\put(120,40){\makebox(0,0){$                    R\frac{1}{\mathbf{f}^{p^{j+1}}}
$}}
\put(-110,40){\vector(1,0){40}}
\put(-90,44){\makebox(0,0){$^{             \subseteq
}$}}
\put(-50,40){\vector(1,0){32}}
\put(-34,44){\makebox(0,0){$^{             \subseteq
}$}}
\put(0,40){\vector(1,0){40}}
\put(137,40){\vector(1,0){40}}
\put(157,44){\makebox(0,0){$^{             \subseteq
}$}}
\put(147,-20){\vector(1,0){30}}
\put(20,44){\makebox(0,0){$^{             \subseteq
}$}}
\put(64,40){\vector(1,0){38}}
\put(83,44){\makebox(0,0){$^{             \subseteq
}$}}
\put(-62,10){\makebox(0,0)[r]{$^{             \cong
}$}}
\put(-58,10){\makebox(0,0)[l]{$^{              \psi_1 = \gamma_1
}$}}
\put(-120,-20){\makebox(0,0){$               R\frac{1}{\mathbf{f}}
$}}
\put(-60,-20){\makebox(0,0){$                F\left(R\frac{1}{\mathbf{f}}\right)
$}}
\put(-7,-20){\makebox(0,0){$                  \cdots
$}}
\put(52,-20){\makebox(0,0){$                    F^j\left(R\frac{1}{\mathbf{f}}\right)
$}}
\put(123,-20){\makebox(0,0){$                   F^{j+1}\left(R\frac{1}{\mathbf{f}}\right)
$}}
\put(-110,-20){\vector(1,0){32}}
\put(-90,-16){\makebox(0,0){$^{             \widetilde{\theta}
}$}}
\put(-43,-20){\vector(1,0){24}}
\put(-31,-16){\makebox(0,0){$^{            F(\widetilde{\theta})
}$}}
\put(0,-20){\vector(1,0){32}}
\put(16,-16){\makebox(0,0){$^{            F^{j-1}( \widetilde{\theta})
}$}}
\put(70,-20){\vector(1,0){26}}
\put(83,-16){\makebox(0,0){$^{            F^j( \widetilde{\theta})
}$}}
\put(-119,30){\line(0,-1){40}}
\put(-121,30){\line(0,-1){40}}
\put(-60,30){\vector(0,-1){40}}
\put(57,30){\vector(0,-1){40}}
\put(55,10){\makebox(0,0)[r]{$^{             \cong
}$}}
\put(59,10){\makebox(0,0)[l]{$^{              \gamma_j
}$}}
\put(120,30){\vector(0,-1){40}}
\put(118,10){\makebox(0,0)[r]{$^{             \cong
}$}}
\put(122,10){\makebox(0,0)[l]{$^{              \gamma_{j+1}
}$}}
\put(187,40){\makebox(0,0){$                      \cdots
$}}
\put(187,-20){\makebox(0,0){$                      \cdots.
$}}
\end{picture}
\]
\item The module of generalized fractions $R_{\mathbf{f}}$ is an $F$-finite $F$-module with $\widetilde{\theta}: R(1/\mathbf{f}) \lra F(R(1/\mathbf{f}))$ as a root morphism. Moreover, the inclusion map $R(1/\mathbf{f}) \stackrel{\subseteq}{\lra} R(1/\mathbf{f}^p)$ is isomorphic to a root of $R_{\mathbf{f}}$.
\end{enumerate}
\end{thm}

\begin{proof} (i) Denote $(f_1^{p^j}, \ldots, f_n^{p^j})^{\lowlim}$ by $\fh_j$. This ideal is, by \ref{gf.103}, the annihilator of the generalized fraction $1/\mathbf{f}^{p^j}$ in $R_{\mathbf{f}}$. Thus there is an isomorphism
$
\phi_j : R(1/\mathbf{f}^{p^{j}}) \stackrel{\cong}{\lra} R/\fh_j
$
for which $\phi_j(r/\mathbf{f}^{p^j}) = r + \fh_j$ for all $r \in R$. Apply $F$ to obtain the isomorphism $F(\phi_j) : F(R(1/\mathbf{f}^{p^j})) \stackrel{\cong}{\lra} F(R/\fh_j)$. But there is an isomorphism $\delta_j : F(R/\fh_j)\stackrel{\cong}{\lra} R/\fh_j^{[p]}$ for which  $\delta_j(r \otimes (1 + \fh_j)) = r + \fh_j^{[p]}$ for all $r\in R$. Now \ref{gf.104} shows that $\fh_j^{[p]} = \fh_{j+1}$, so that there is the isomorphism $\phi_{j+1} : R(1/\mathbf{f}^{p^{j+1}})\stackrel{\cong}{\lra} R/\fh_{j+1} = R/\fh_j^{[p]}$.  The composite isomorphism $\psi_{j+1} := (F(\phi_j))^{-1}\circ\delta_j^{-1} \circ \phi_{j+1} : R(1/\mathbf{f}^{p^{j+1}}) \stackrel{\cong}{\lra} F(R(1/\mathbf{f}^{p^{j}}))$
satisfies
\begin{align*}
\psi_{j+1}\big( r/\mathbf{f}^{p^{j+1}}\big)& = (F(\phi_j))^{-1}\circ\delta_j^{-1}(r + \fh_{j+1})= (F(\phi_j))^{-1}\circ\delta_j^{-1}(r + \fh_j^{[p]}) \\&= (F(\phi_j))^{-1}(r \otimes (1 + \fh_j)) = r \otimes (1/\mathbf{f}^{p^j}) \quad \mbox{for all $r \in R$.}
\end{align*}

(ii) We interpret $F^0$ as the identity functor (on the category of $R$-modules) and define $\gamma_0$ to be the identity mapping on $R(1/\mathbf{f})$. Define $\gamma_1:= F(\gamma_0)\circ \psi_1 = \psi_1 : R(1/\mathbf{f}^p) \stackrel{\cong}{\lra} F(R(1/\mathbf{f}))$, and use this as the basis for a straightforward induction employing the formula $\gamma_{j+1} := F(\gamma_j) \circ \psi_{j+1}$ for all $j \in \N$.

(iii) It is clear that the left-most square in the diagram commutes and that $\widetilde{\theta}$ is a monomorphism. Also, for $r \in R$,
\[
\widetilde{\theta}(r/\mathbf{f}) = \psi_1(f_1^{p-1}\ldots f_n^{p-1}r/\mathbf{f}^p) = f_1^{p-1}\ldots f_n^{p-1}r \otimes 1/\mathbf{f}.
\]

Consider a square in the diagram with $j \geq 1$. Let $r \in R$. Then
\begin{align*}
\gamma_{j+1}\left(\frac{r}{\mathbf{f}^{p^j}}\right)& = \gamma_{j+1}\left(\frac{f_1^{p^{j+1}- p^j}\ldots f_n^{p^{j+1}- p^j}r}{\mathbf{f}^{p^{j+1}}}\right)\\ &= f_1^{p^{j+1}- p^j}\ldots f_n^{p^{j+1}- p^j}r \otimes\Big( 1 \otimes \Big( \cdots \otimes \Big(1 \otimes \frac{1}{\mathbf{f}}\Big)\cdots\Big)\Big)\quad \mbox{($j+1$ tensor products)},
\end{align*}
whereas
\begin{align*}
F^j(\widetilde{\theta})\circ\gamma_{j}\left(\frac{r}{\mathbf{f}^{p^j}}\right)& = F^j(\widetilde{\theta})\Big(r \otimes\Big( 1 \otimes \Big( \cdots \otimes \Big(1 \otimes \frac{1}{\mathbf{f}}\Big)\cdots\Big)\Big)\Big) \quad \mbox{($j$ tensor products)}\\ &=
r \otimes\Big( 1 \otimes \Big( \cdots \otimes \Big(f_1^{p-1}\ldots f_n^{p-1} \otimes \frac{1}{\mathbf{f}}\Big)\cdots\Big)\Big)  \quad \mbox{($j+1$ tensor products)}\\ &= r \otimes\Big( 1 \otimes \Big( \cdots \otimes \Big(f_1^{p^2-p}\ldots f_n^{p^2-p} \otimes \Big(1 \otimes \frac{1}{\mathbf{f}}\Big)\Big)\cdots\Big)\Big)  \quad \mbox{($j+1$ tensor products)}
\\ &= \cdots = f_1^{p^{j+1}- p^j}\ldots f_n^{p^{j+1}- p^j}r \otimes\Big( 1 \otimes \Big( \cdots \otimes \Big(1 \otimes \frac{1}{\mathbf{f}}\Big)\cdots\Big)\Big)= \gamma_{j+1}\Big(\frac{r}{\mathbf{f}^{p^j}}\Big).
\end{align*}
Thus the diagram in the statement of the theorem commutes.

(iv) The top row in the diagram gives rise to a direct system whose direct limit is $R_{\mathbf{f}}$. The bottom row gives rise to a direct system, whose direct limit we denote by ${\mathcal M}$; it is clear that ${\mathcal M} \cong R_{\mathbf{f}}$. Recall from \ref{GLF.1} that the map ${\mathcal M} \lra F({\mathcal M})$ induced by the commutative diagram
\[
\begin{picture}(300,75)(-150,-25)
\put(-140,40){\makebox(0,0){$                   R\frac{1}{\mathbf{f}}
$}}
\put(-70,40){\makebox(0,0){$                    F(R\frac{1}{\mathbf{f}})
$}}
\put(-7,40){\makebox(0,0){$                      \cdots
$}}
\put(62,40){\makebox(0,0){$                     F^i(R\frac{1}{\mathbf{f}})
$}}
\put(140,40){\makebox(0,0){$                    F^{i+1}(R\frac{1}{\mathbf{f}})
$}}
\put(-129,40){\vector(1,0){40}}
\put(-109,44){\makebox(0,0){$^{             \widetilde{\theta}
}$}}
\put(-54,40){\vector(1,0){36}}
\put(-36,44){\makebox(0,0){$^{             F(\widetilde{\theta})
}$}}
\put(0,40){\vector(1,0){40}}
\put(20,44){\makebox(0,0){$^{             F^{i-1}(\widetilde{\theta})
}$}}
\put(79,40){\vector(1,0){38}}
\put(98,44){\makebox(0,0){$^{             F^i(\widetilde{\theta})
}$}}
\put(-68,10){\makebox(0,0)[l]{$^{              F(\widetilde{\theta})
}$}}
\put(-140,-20){\makebox(0,0){$               F(R\frac{1}{\mathbf{f}})
$}}
\put(-70,-20){\makebox(0,0){$                F^2(R\frac{1}{\mathbf{f}})
$}}
\put(-7,-20){\makebox(0,0){$                  \cdots
$}}
\put(62,-20){\makebox(0,0){$                    F^{i+1}(R\frac{1}{\mathbf{f}})
$}}
\put(143,-20){\makebox(0,0){$                   F^{i+2}(R\frac{1}{\mathbf{f}})
$}}
\put(-125,-20){\vector(1,0){34}}
\put(-108,-16){\makebox(0,0){$^{             F(\widetilde{\theta})
}$}}
\put(-53,-20){\vector(1,0){34}}
\put(-36,-16){\makebox(0,0){$^{            F^2(\widetilde{\theta})
}$}}
\put(6,-20){\vector(1,0){32}}
\put(85,-20){\vector(1,0){34}}
\put(102,-16){\makebox(0,0){$^{            F^{i+1}(\widetilde{\theta})
}$}}
\put(-140,30){\vector(0,-1){40}}
\put(-134,10){\makebox(0,0){$^{             \widetilde{\theta}
}$}}
\put(-70,30){\vector(0,-1){40}}
\put(57,30){\vector(0,-1){40}}
\put(59,10){\makebox(0,0)[l]{$^{              F^{i}(\widetilde{\theta})
}$}}
\put(140,30){\vector(0,-1){40}}
\put(142,10){\makebox(0,0)[l]{$^{              F^{i+1}(\widetilde{\theta})
}$}}
\put(199,40){\makebox(0,0){$                      \cdots
$}}
\put(164,-20){\vector(1,0){24}}
\put(164,40){\vector(1,0){24}}
\put(199,-20){\makebox(0,0){$                      \cdots
$}}
\end{picture}
\]
is an isomorphism. Therefore $R_{\mathbf{f}}$ is an $F$-finite $F$-module with $\widetilde{\theta}: R(1/\mathbf{f}) \lra F(R(1/\mathbf{f}))$ as a root morphism. The final claim follows from the definition of $\widetilde{\theta}$.
\end{proof}

In the remainder of this section, we shall show that, if ${\mathcal M}$ is any $F$-finite $F$-module, then the module of generalized fractions ${\mathcal M}_{\mathbf{f}}$ is an $F$-finite $F$-module. It will follow that, if ${\mathbf g_1},\ldots,{\mathbf g_t}$ are finite sequences of elements of $R$ (possibly of different lengths), then the $R$-module
$$
(\ldots(({\mathcal M}_{\mathbf{g_1}})_{\mathbf{g_2}})\ldots)_{\mathbf{g_t}},
$$
produced by $t$ successive constructions of modules of generalized fractions, is again an $F$-finite $F$-module. In particular, since $R$ itself is an $F$-finite $F$-module,
$
(\ldots((R_{\mathbf{g_1}})_{\mathbf{g_2}})\ldots)_{\mathbf{g_t}}
$
is an $F$-finite $F$-module.

In the situation of \ref{add.1}, that lemma yields an isomorphism $F(X) \otimes_R F(Y) \stackrel{\cong}{\lra} F(X \otimes_R Y)$, and application of the functor $F$ produces an isomorphism $F(F(X) \otimes_R F(Y)) \stackrel{\cong}{\lra} F^2(X \otimes_R Y)$.  However, application of \ref{add.1} to the $R$-modules $F(X)$ and $F(Y)$ yields an isomorphism $F^2(X) \otimes_R F^2(Y) \stackrel{\cong}{\lra} F(F(X) \otimes_R F(Y))$ and composition of the latter two isomorphisms yields an isomorphism $$F^2(X) \otimes_R F^2(Y) \stackrel{\cong}{\lra} F^2(X \otimes_R Y).$$ It will be convenient to have some formal terminology and notation.

\begin{ntn}
\label{add.2} Let the situation and notation be as in \ref{add.1}. We refer to the isomorphism $$
\Delta : F(X) \otimes_R F(Y) \stackrel{\cong}{\lra} F(X \otimes_R Y)
$$ of \ref{add.1} as the {\em canonical isomorphism}. When there is need for greater precision, we shall denote $\Delta$ by $\Delta^1_{X,Y}$.

Suppose, inductively, that $i \in \N$ and the isomorphism $\Delta^i_{X,Y} : F^i(X) \otimes_R F^i(Y) \stackrel{\cong}{\lra} F^i(X \otimes_R Y)$ has been defined for all choices of $R$-modules $X$ and $Y$. Let $\Delta^{i+1}_{X,Y} : F^{i+1}(X) \otimes_R F^{i+1}(Y) \stackrel{\cong}{\lra} F^{i+1}(X \otimes_R Y)$ be the composition of the isomorphisms
$$
\Delta^1_{F^i(X),F^i(Y)} : F^{i+1}(X) \otimes_R F^{i+1}(Y) \stackrel{\cong}{\lra} F(F^i(X) \otimes_R F^i(Y))
$$
and
$$
F(\Delta^i_{X,Y}) : F(F^i(X) \otimes_R F^i(Y)) \stackrel{\cong}{\lra} F(F^i(X \otimes_R Y)) = F^{i+1}(X \otimes_R Y).
$$
This induction defines the isomorphism $\Delta^j_{X,Y} : F^j(X) \otimes_R F^j(Y) \stackrel{\cong}{\lra} F^j(X \otimes_R Y)$, called the {\em canonical isomorphism}, for all $j \in \N$. We shall occasionally extend the notation and use $\Delta^0_{X,Y} : X \otimes_R Y\lra X \otimes_R Y$ to denote the identity map on $X \otimes_R Y$.
\end{ntn}

\begin{lem}
\label{add.2a} For arbitrary $R$-modules $X$ and $Y$, and $R$-homomorphisms $\lambda : X \lra F(X)$ and $\mu : Y \lra F(Y)$, the diagram
\[
\begin{picture}(300,75)(-150,-25)
\put(-25,40){\makebox(0,0)[r]{$ F(X)\otimes_RF(Y)
$}} \put(25,40){\makebox(0,0)[l]{$ F^2(X) \otimes_R F^2(Y) $}} \put(0,44){\makebox(0,0){$^{
F(\lambda) \otimes F(\mu)}$}} \put(-47,10){\makebox(0,0)[l]{$^{ \cong }$}}
\put(-24,40){\vector(1,0){48}} \put(-51,10){\makebox(0,0)[r]{$^{
\Delta^1_{X,Y} }$}}
\put(52,10){\makebox(0,0)[l]{$^{ \Delta^1_{F(X),F(Y)} }$}} \put(-25,-20){\makebox(0,0)[r]{$ F(X\otimes_RY) $}} \put(25,-20){\makebox(0,0)[l]{$
F(F(X)\otimes_RF(Y))$}}
\put(0,-16){\makebox(0,0){$^{ F(\lambda \otimes \mu) }$}}
\put(47,10){\makebox(0,0)[r]{$^{ \cong }$}}
\put(-22,-20){\vector(1,0){44}} \put(-50,30){\vector(0,-1){40}}
\put(50,30){\vector(0,-1){40}}
\end{picture}
\]
commutes, where $\Delta^1_{X,Y}$ and $\Delta^1_{F(X),F(Y)}$ are as defined in\/ {\rm \ref{add.2}}.
\end{lem}

\begin{proof}[Proof.] Let $x \in X, y \in Y$ and $a',b'\in R'$; then
\begin{align*} \Delta^1_{F(X),F(Y)}\circ (F(\lambda) \otimes F(\mu))((a' \otimes x)\otimes(b' \otimes y)) & = \Delta^1_{F(X),F(Y)}((a' \otimes \lambda(x))\otimes(b' \otimes \mu(y))) \\ &= a'b'\otimes (\lambda(x)\otimes\mu(y)) =
F(\lambda \otimes \mu)(a'b'\otimes (x \otimes y)) \\ &= F(\lambda \otimes \mu)\circ \Delta^1_{X,Y}((a'\otimes x) \otimes (b'\otimes y)). \qedhere
\end{align*}
\end{proof}

\begin{disc}
\label{add.3} Assume that the ring $A$ is Noetherian and suppose we have a sequence
$$
X^0 \stackrel{f_0}{\lra} X^1 \lra \cdots \lra X^i \stackrel{f_i}{\lra} X^{i+1} \lra \cdots
$$
of $A$-modules and $A$-homomorphisms. By the {\em associated direct system\/} we shall mean the direct system indexed by the set $\nn$ of non-negative integers whose constituent modules are $X^0, X^1, \ldots, X^i, \ldots$ and whose constituent homomorphism for $i < k$ in $\nn$ is the composite homomorphism $f_{k-1} \circ \cdots \circ f_i$. Let $X^{\infty}$ denote the direct limit of this associated direct system.

Now suppose we have a second sequence
$$
Y^0 \stackrel{g_0}{\lra} Y^1 \lra \cdots \lra Y^i \stackrel{g_i}{\lra} Y^{i+1} \lra \cdots
$$
of $A$-modules and $A$-homomorphisms and that the associated direct system has direct limit $Y^{\infty}$.

Fix $i \in \nn$. The fact that tensor product commutes with direct limits means that $X^i \otimes_A Y^{\infty}$ is the direct limit of the direct system associated to the sequence
$$
X^i \otimes_AY^0 \stackrel{X^i \otimes g_0}{\lra} X^i \otimes_AY^1 \lra \cdots \lra X^i \otimes_AY^j \stackrel{X^i \otimes g_j}{\lra} X^i \otimes_AY^{j+1} \lra \cdots
$$
and $X^{\infty} \otimes_A Y^{\infty}$ is the direct limit of the direct system associated to the sequence
$$
X^0 \otimes_AY^{\infty} \stackrel{f_0 \otimes Y^{\infty}}{\lra} X^1 \otimes_AY^{\infty} \lra \cdots \lra X^i \otimes_AY^{\infty} \stackrel{f_i \otimes Y^{\infty}}{\lra} X^{i+1} \otimes_AY^{\infty} \lra \cdots.
$$

A straightforward argument involving these facts will show that $X^{\infty} \otimes_A Y^{\infty}$ is the direct limit of the direct system associated to the `diagonal' sequence
$$
X^0 \otimes_AY^0 \stackrel{f_0 \otimes g_0}{\lra} X^1 \otimes_AY^1 \lra \cdots \lra X^i \otimes_AY^i \stackrel{f_i \otimes g_i}{\lra} X^{i+1} \otimes_AY^{i+1} \lra \cdots.
$$
\end{disc}

For the remainder of this section, we shall work over $R$: see \ref{GLF.1}.

\begin{thm}
\label{add.4} Let ${\mathcal M}$ be an $F$-finite $F$-module over $R$ with generating morphism $\alpha : M \lra F(M)$, where $M$ is a finitely generated $R$-module. Also, let ${\mathcal N}$ be a second $F$-finite $F$-module over $R$ with generating morphism $\beta : N \lra F(N)$, where $N$ is a finitely generated $R$-module.

Then ${\mathcal M} \otimes_R {\mathcal N}$ is an $F$-finite $F$-module with generating morphism
$$
\gamma := \Delta \circ (\alpha \otimes \beta): M \otimes_R N \lra F(M\otimes_R N),
$$
where $\Delta : F(M) \otimes_R F(N) \stackrel{\cong}{\lra} F(M\otimes_R N)$ is the isomorphism given by Lemma\/ {\rm \ref{add.1}} (and denoted by $\Delta^1_{M,N}$ in\/ {\rm \ref{add.2}}).
\end{thm}

\begin{proof}[Proof.] We can assume that ${\mathcal M}$ is the direct limit of the direct system associated to the sequence
\[
\begin{picture}(300,15)(-150,35)
\put(-120,40){\makebox(0,0){$                   M
$}}
\put(-60,40){\makebox(0,0){$                    F(M)
$}}
\put(-7,40){\makebox(0,0){$                      \cdots
$}}
\put(52,40){\makebox(0,0){$                     F^i(M)
$}}
\put(120,40){\makebox(0,0){$                    F^{i+1}(M)
$}}
\put(-110,40){\vector(1,0){36}}
\put(-90,44){\makebox(0,0){$^{             \alpha
}$}}
\put(-47,40){\vector(1,0){29}}
\put(-34,44){\makebox(0,0){$^{             F(\alpha)
}$}}
\put(0,40){\vector(1,0){36}}
\put(20,44){\makebox(0,0){$^{             F^{i-1}(\alpha)
}$}}
\put(67,40){\vector(1,0){32}}
\put(83,44){\makebox(0,0){$^{             F^i(\alpha)
}$}}
\put(142,40){\vector(1,0){28}}
\put(180,40){\makebox(0,0){$                      \cdots
$}}
\end{picture}
\]
and that ${\mathcal N}$ is the direct limit of the direct system associated to the sequence
\[
\begin{picture}(300,15)(-150,35)
\put(-120,40){\makebox(0,0){$                   N
$}}
\put(-60,40){\makebox(0,0){$                    F(N)
$}}
\put(-7,40){\makebox(0,0){$                      \cdots
$}}
\put(52,40){\makebox(0,0){$                     F^i(N)
$}}
\put(120,40){\makebox(0,0){$                    F^{i+1}(N)
$}}
\put(-110,40){\vector(1,0){36}}
\put(-90,44){\makebox(0,0){$^{             \beta
}$}}
\put(-47,40){\vector(1,0){29}}
\put(-34,44){\makebox(0,0){$^{             F(\beta)
}$}}
\put(0,40){\vector(1,0){36}}
\put(20,44){\makebox(0,0){$^{             F^{i-1}(\beta)
}$}}
\put(67,40){\vector(1,0){32}}
\put(83,44){\makebox(0,0){$^{             F^i(\beta)
}$}}
\put(142,40){\vector(1,0){28}}
\put(180,40){\makebox(0,0){$                      \cdots.
$}}
\end{picture}
\]
For each $i \in \nn$, let $L^i := F^i(M) \otimes_RF^i(N)$ and let $d^i := F^i(\alpha) \otimes F^i(\beta) : F^i(M) \otimes_RF^i(N) = L^i \lra F^{i+1}(M) \otimes_RF^{i+1}(N) = L^{i+1}$. It follows from \ref{add.3} that ${\mathcal M} \otimes_R {\mathcal N}$ is isomorphic to the direct limit of the direct system associated to the `diagonal' sequence
$$
M \otimes_RN \stackrel{d^0}{\lra} F(M) \otimes_RF(N) \lra \cdots \lra F^i(M) \otimes_RF^i(N) \stackrel{d^i}{\lra} F^{i+1}(M) \otimes_RF^{i+1}(N) \lra \cdots.
$$

Observe that the diagram
\[
\begin{picture}(300,75)(-150,-25)
\put(-25,40){\makebox(0,0)[r]{$ M\otimes_RN
$}} \put(25,40){\makebox(0,0)[l]{$ F(M) \otimes_R F(N) $}} \put(0,44){\makebox(0,0){$^{
\alpha \otimes \beta}$}}
\put(-20,40){\vector(1,0){40}}
\put(53,10){\makebox(0,0)[r]{$^{ \Delta^1_{M,N} }$}} \put(-25,-20){\makebox(0,0)[r]{$ M\otimes_RN $}} \put(25,-20){\makebox(0,0)[l]{$
F(M\otimes_RN)$}}
\put(0,-16){\makebox(0,0){$^{ \gamma }$}}
\put(57,10){\makebox(0,0)[l]{$^{ \cong }$}}
\put(-20,-20){\vector(1,0){40}} \put(-46,30){\line(0,-1){40}} \put(-48,30){\line(0,-1){40}}
\put(54,30){\vector(0,-1){40}}
\end{picture}
\]
commutes. This fact is a basis for our inductive proof that, for all $i \in \nn$, the diagram
\[
\begin{picture}(300,75)(-150,-25)
\put(-30,40){\makebox(0,0)[r]{$ F^i(M)\otimes_RF^i(N)
$}} \put(30,40){\makebox(0,0)[l]{$ F^{i+1}(M) \otimes_R F^{i+1}(N) $}} \put(0,46){\makebox(0,0){$^{
F^i(\alpha) \otimes F^i(\beta)}$}} \put(-62,10){\makebox(0,0)[l]{$^{ \cong }$}}
\put(-27,40){\vector(1,0){54}} \put(-67,10){\makebox(0,0)[r]{$^{
\Delta^i_{M,N} }$}}
\put(63,10){\makebox(0,0)[r]{$^{ \Delta^{i+1}_{M,N} }$}} \put(-30,-20){\makebox(0,0)[r]{$ F^i(M\otimes_RN) $}} \put(30,-20){\makebox(0,0)[l]{$
F^{i+1}(M\otimes_RN)$}}
\put(0,-14){\makebox(0,0){$^{ F^i(\gamma) }$}}
\put(68,10){\makebox(0,0)[l]{$^{ \cong }$}}
\put(-27,-20){\vector(1,0){54}} \put(-65,30){\vector(0,-1){40}}
\put(65,30){\vector(0,-1){40}}
\end{picture}
\]
commutes. Suppose that the above diagram does indeed commute for some $i \in \nn$. Apply the functor $F$ to see that the diagram
\[
\begin{picture}(300,75)(-150,-25)
\put(-40,40){\makebox(0,0)[r]{$ F(F^i(M)\otimes_RF^i(N))
$}} \put(40,40){\makebox(0,0)[l]{$ F(F^{i+1}(M) \otimes_R F^{i+1}(N)) $}} \put(0,46){\makebox(0,0){$^{
F(F^i(\alpha) \otimes F^i(\beta))}$}} \put(-77,10){\makebox(0,0)[l]{$^{ \cong }$}}
\put(-35,40){\vector(1,0){70}} \put(-82,10){\makebox(0,0)[r]{$^{
F(\Delta^i_{M,N})}$}}
\put(78,10){\makebox(0,0)[r]{$^{F( \Delta^{i+1}_{M,N}) }$}} \put(-40,-20){\makebox(0,0)[r]{$ F^{i+1}(M\otimes_RN) $}} \put(40,-20){\makebox(0,0)[l]{$
F^{i+2}(M\otimes_RN)$}}
\put(0,-14){\makebox(0,0){$^{ F^{i+1}(\gamma) }$}}
\put(83,10){\makebox(0,0)[l]{$^{ \cong }$}}
\put(-35,-20){\vector(1,0){70}} \put(-80,30){\vector(0,-1){40}}
\put(80,30){\vector(0,-1){40}}
\end{picture}
\]
commutes.

By \ref{add.2a} applied to the $R$-modules $F^i(M)$ and $F^i(N)$ (and the homomorphisms $F^i(\alpha): F^i(M) \lra F^{i+1}(M)$ and $F^i(\beta): F^i(N) \lra F^{i+1}(N)$), the diagram
\[
\begin{picture}(300,75)(-150,-25)
\put(-40,40){\makebox(0,0)[r]{$ F^{i+1}(M)\otimes_RF^{i+1}(N)
$}} \put(40,40){\makebox(0,0)[l]{$ F^{i+2}(M) \otimes_R F^{i+2}(N) $}} \put(0,46){\makebox(0,0){$^{
F^{i+1}(\alpha) \otimes F^{i+1}(\beta)}$}} \put(-87,10){\makebox(0,0)[l]{$^{ \cong }$}}
\put(-35,40){\vector(1,0){70}} \put(-92,10){\makebox(0,0)[r]{$^{
\Delta^1_{F^i(M),F^i(N)} }$}}
\put(88,10){\makebox(0,0)[r]{$^{ \Delta^1_{F^{i+1}(M),F^{i+1}(N)} }$}} \put(-40,-20){\makebox(0,0)[r]{$ F(F^i(M)\otimes_RF^i(N)) $}} \put(40,-20){\makebox(0,0)[l]{$
F(F^{i+1}(M)\otimes_RF^{i+1}(N))$}}
\put(0,-14){\makebox(0,0){$^{ F(F^i(\alpha) \otimes F^i(\beta)) }$}}
\put(93,10){\makebox(0,0)[l]{$^{ \cong }$}}
\put(-35,-20){\vector(1,0){70}} \put(-90,30){\vector(0,-1){40}}
\put(90,30){\vector(0,-1){40}}
\end{picture}
\]
commutes. Now splice the last two commutative diagrams together, and recall from \ref{add.2} the definitions of $\Delta^{i+1}_{M,N}$ and
$\Delta^{i+2}_{M,N}$.  We can conclude that the diagram
\[
\begin{picture}(300,75)(-150,-25)
\put(-40,40){\makebox(0,0)[r]{$ F^{i+1}(M)\otimes_RF^{i+1}(N)
$}} \put(40,40){\makebox(0,0)[l]{$ F^{i+2}(M) \otimes_R F^{i+2}(N) $}} \put(0,46){\makebox(0,0){$^{
F^{i+1}(\alpha) \otimes F^{i+1}(\beta)}$}} \put(-77,10){\makebox(0,0)[l]{$^{ \cong }$}}
\put(-35,40){\vector(1,0){70}} \put(-82,10){\makebox(0,0)[r]{$^{
\Delta^{i+1}_{M,N} }$}}
\put(78,10){\makebox(0,0)[r]{$^{ \Delta^{i+2}_{M,N} }$}} \put(-40,-20){\makebox(0,0)[r]{$ F^{i+1}(M\otimes_RN) $}} \put(40,-20){\makebox(0,0)[l]{$
F^{i+2}(M\otimes_RN)$}}
\put(0,-14){\makebox(0,0){$^{ F^{i+1}(\gamma) }$}}
\put(83,10){\makebox(0,0)[l]{$^{ \cong }$}}
\put(-35,-20){\vector(1,0){70}} \put(-80,30){\vector(0,-1){40}}
\put(80,30){\vector(0,-1){40}}
\end{picture}
\]
commutes. This completes our inductive argument that shows that the diagram
\[
\begin{picture}(300,75)(-150,-25)
\put(-30,40){\makebox(0,0)[r]{$ F^i(M)\otimes_RF^i(N)
$}} \put(30,40){\makebox(0,0)[l]{$ F^{i+1}(M) \otimes_R F^{i+1}(N) $}} \put(0,46){\makebox(0,0){$^{
F^i(\alpha) \otimes F^i(\beta)}$}} \put(-62,10){\makebox(0,0)[l]{$^{ \cong }$}}
\put(-27,40){\vector(1,0){54}} \put(-67,10){\makebox(0,0)[r]{$^{
\Delta^i_{M,N} }$}}
\put(63,10){\makebox(0,0)[r]{$^{ \Delta^{i+1}_{M,N} }$}} \put(-30,-20){\makebox(0,0)[r]{$ F^i(M\otimes_RN) $}} \put(30,-20){\makebox(0,0)[l]{$
F^{i+1}(M\otimes_RN)$}}
\put(0,-14){\makebox(0,0){$^{ F^i(\gamma) }$}}
\put(68,10){\makebox(0,0)[l]{$^{ \cong }$}}
\put(-27,-20){\vector(1,0){54}} \put(-65,30){\vector(0,-1){40}}
\put(65,30){\vector(0,-1){40}}
\end{picture}
\]
commutes for all $i \in \nn$. It follows that ${\mathcal M} \otimes_R {\mathcal N}$ is isomorphic to the direct limit of the direct system associated to the sequence\[
\begin{picture}(300,15)(-150,35)
\put(-190,40){\makebox(0,0){$                 M\otimes_R N
$}}
\put(-110,40){\makebox(0,0){$                    F(M\otimes_RN)
$}}
\put(-40,40){\makebox(0,0){$                      \cdots
$}}
\put(40,40){\makebox(0,0){$                     F^i(M\otimes_RN)
$}}
\put(140,40){\makebox(0,0){$                    F^{i+1}(M\otimes_RN)
$}}
\put(-170,40){\vector(1,0){30}}
\put(-155,44){\makebox(0,0){$^{             \gamma
}$}}
\put(-80,40){\vector(1,0){30}}
\put(-65,44){\makebox(0,0){$^{             F(\gamma)
}$}}
\put(-32,40){\vector(1,0){40}}
\put(-10,44){\makebox(0,0){$^{             F^{i-1}(\gamma)
}$}}
\put(72,40){\vector(1,0){32}}
\put(88,44){\makebox(0,0){$^{             F^i(\gamma)
}$}}
\put(176,40){\vector(1,0){20}}
\put(210,40){\makebox(0,0){$                      \cdots.
$}}
\end{picture}
\]
Since $M \otimes_RN$ is a finitely generated $R$-module, we deduce that ${\mathcal M} \otimes_R {\mathcal N}$ is an $F$-finite $F$-module with generating morphism
$
\gamma := \Delta^1_{M,N} \circ (\alpha \otimes \beta): M \otimes_R N \lra F(M\otimes_R N).
$
\end{proof}

\begin{cor}
\label{add.5} Let ${\mathcal M}$ be an $F$-finite $F$-module over $R$. Let $n \in \N$ and let $\mathbf{f} := (f_1, \ldots, f_n) \in R^n$. Then the module of generalized fractions ${\mathcal M}_{\mathbf{f}} := U_{\mathbf{f}}^{-n}{\mathcal M}$ is an $F$-finite $F$-module.
\end{cor}

\begin{proof}[Proof.] This is immediate from \ref{add.4} once it is recalled from \ref{calc.6} that ${\mathcal M}_{\mathbf{f}} := U_{\mathbf{f}}^{-n}{\mathcal M} \cong {\mathcal M} \otimes_RU_{\mathbf{f}}^{-n}(R)$ and from \ref{gf.107}(iv) that $U_{\mathbf{f}}^{-n}(R)$ is an $F$-finite $F$-module.
\end{proof}

\begin{cor}
\label{add.6} Let ${\mathcal M}$ be an $F$-finite $F$-module over $R$. If ${\mathbf g_1},\ldots,{\mathbf g_t}$ are finite sequences of elements of $R$ (possibly of different lengths), then the $R$-module
$$
(\ldots(({\mathcal M}_{\mathbf{g_1}})_{\mathbf{g_2}})\ldots)_{\mathbf{g_t}},
$$
produced by $t$ successive constructions of modules of generalized fractions, is again an $F$-finite $F$-module. In particular,
$
(\ldots((R_{\mathbf{g_1}})_{\mathbf{g_2}})\ldots)_{\mathbf{g_t}}
$
is an $F$-finite $F$-module.
\end{cor}

\begin{proof}[Proof.] This is immediate from \ref{add.5}; note that $R$ itself is an $F$-finite $F$-module.
\end{proof}

\begin{lem}
\label{alg.5} Let ${\mathcal M}$ be an $F$-finite $F$-module over $R$ with root $N$. The following hold:
\begin{enumerate}
\item $\mu^i(\fp,F(N)) \leq \mu^i(\fp,N)$ for all $i \in \nn$ and $\fp \in \Spec(R)$;
\item if $\fm$ is a maximal ideal of $R$, then $\mu^0(\fm,{\mathcal M}) = \mu^0(\fm,N)= \mu^0(\fm,F(N))$.
\end{enumerate}
\end{lem}

\begin{proof}[Proof.] There is an $R$-monomorphism $\beta : N \lra F(N)$ such that ${\mathcal M}$ is the direct limit of the resulting direct system $(F^i(N))_{i\in \nn}$.

(i) Let the exact sequence
\[
0 \lra N \lra E^0(N) \stackrel{d^0}{\lra} E^{1}(N) \lra \cdots \lra E^{i}(N) \lra \cdots
\]
provide the injective resolution of $N$, so that, for each $i \in \nn$,
\[
E^i(N) \cong \bigoplus_{\fp \in \Spec(R)} E(R/\fp)^{\mu^i(\fp,N)},
\]
where $ E^{\mu}$ denotes a direct sum of $\mu$ copies of $E$. Apply the exact functor $F$: since $F(E) \cong E$ for each injective $R$-module $E$ (by \cite[Proposition 1.5]{58}), the exact sequence
\[
0 \lra F(N) \lra F(E^0(N)) \stackrel{F(d^0)}{\lra} F(E^{1}(N)) \lra \cdots \lra F(E^{i}(N)) \lra \cdots
\]
provides an injective resolution of $F(N)$, so that $F(E^{i}(N))$ has a direct summand isomorphic to $E^{i}(F(N))$, for each  $i \in \nn$ (by \cite[11.1.11]{LC}, for example). Another use of \cite[Proposition 1.5]{58} now shows that
\[
F(E^i(N)) \cong \bigoplus_{\fp \in \Spec(R)} F(E(R/\fp))^{\mu^i(\fp,N)}\cong \bigoplus_{\fp \in \Spec(R)} E(R/\fp)^{\mu^i(\fp,N)},
\]
for each $i \in \nn$. It follows that $\mu^i(\fp,F(N)) \leq \mu^i(\fp,N)$ for all $i \in \nn$ and $\fp \in \Spec(R)$.

(ii) By part (i) above, $\mu^0(\fm,F(N)) \leq \mu^0(\fm,N)$. Write $\mu^0(\fm,N) = h$; then $E^0(N)$ has a submodule isomorphic to $ (R/\fm)^h$, so that, because $\fm$ is maximal, $N$ has a submodule isomorphic to $ (R/\fm)^h$. Apply the exact functor $F$: we see that $F(N)$ has a submodule isomorphic to $(R/\fm^{[p]})^h$. But $\fm$ is the unique minimal prime ideal of $\fm^{[p]}$, and so $R/\fm^{[p]}$ has a submodule isomorphic to $R/\fm$. Thus $F(N)$ has a submodule isomorphic to $ (R/\fm)^h$ and $\mu^0(\fm,F(N)) \geq h$. Therefore $\mu^0(\fm,N) = \mu^0(\fm,F(N)) = \mu^0(\fm,F^k(N))$ for all $k \in \N$. A minor modification of \cite[1.7]{58} now enables us to see that $\mu^0(\fm,{\mathcal M}) \leq h$. On the other hand, the fact that $N$ can be embedded in ${\mathcal M}$ ensures that $\mu^0(\fm,{\mathcal M}) \geq h$.
\end{proof}

\section{\bf Filter-regular sequences}
\label{frs}

Throughout this section, we shall assume that the commutative ring $A$ is Noetherian. For most of the section, we shall work over $A$;  we shall not make any assumption about the characteristic of $A$, and we shall not assume that $A$ is regular. Also, $\fb$ will denote an ideal of $A$. The {\em variety $\Var(\fb)$ of $\fb$\/} is the set $\{\fp \in \Spec(A) : \fp \supseteq \fb\}$. We shall use $M$ to denote an arbitrary $A$-module; $M$ will only be assumed to be finitely generated when this is explicitly stated. We say that $b_1, \ldots, b_n \in A$ form (or is) a {\em poor $M$-sequence\/} if $b_i$ is a non-zero-divisor on $M/\sum_{j=1}^{i-1} b_jM$ for all $i = 1, \ldots, n$. We say that an $A$-module $N$ is {\em $\fb$-torsion\/} if each element of $N$ is annihilated by some power of $\fb$; note that this is the case if and only if $\Supp(N) \subseteq \Var(\fb)$.

\begin{defi}
\label{defilt.1} Let $a_1 ,\ldots, a_n \in A$.  We say that $a_1, \ldots, a_n$ form (or is) a {\em $\fb$-filter-regular sequence on $M$\/} (of length $n$) if
\[
\Supp\big(  \big({\textstyle \sum_{j=1}^{i-1} a_jM:_Ma_i}\big)\big/\big({\textstyle \sum_{j=1}^{i-1} a_jM}\big)\big) \subseteq \Var(\fb)  \quad \mbox{for all~} i = 1, \ldots, n.
\]
We regard the empty sequence as a $\fb$-filter-regular sequence on $M$ of length $0$. We say that an infinite sequence $(a_i)_{i\in\N}$ of elements of $A$ is a {\em $\fb$-filter-regular sequence on $M$\/} if $a_1, \ldots, a_t$ is a $\fb$-filter-regular sequence on $M$ for all $t \in \N$. If $a_1 ,\ldots, a_n$ is a $\fb$-filter-regular sequence on $M$, then $a_1 ,\ldots, a_n, 1, 1, \ldots$ is an infinite $\fb$-filter-regular sequence on $M$.

The reader should note that in this definition we have not required the elements of a $\fb$-filter-regular sequence on $M$ to be members of $\fb$. We think there are advantages in this approach, even though several authors have required that the elements of an $\fm$-filter-regular sequence on a local ring $(R,\fm)$ belong to $\fm$.
\end{defi}

The proof of the following lemma is elementary and left to the reader.

\begin{lem}
\label{defilt.2} Let $a_1 ,\ldots, a_n\in A$.  Then the following statements are equivalent:
\begin{enumerate}
\item $a_1 ,\ldots, a_n$ is a $\fb$-filter-regular sequence on $M$;
\item for all $\fp \in \Spec(A) \setminus \Var(\fb)$, the sequence $a_1/1 ,\ldots, a_n/1$ of natural images in $A_{\fp}$ is a poor $M_{\fp}$-sequence;
\item for all $i = 1, \ldots, n$, the $A$-module $\left(\sum_{j=1}^{i-1} a_jM:_Ma_i\right)\Big/\left(\sum_{j=1}^{i-1} a_jM\right)$ is $\fb$-torsion.
\end{enumerate}
\end{lem}

\begin{rmk}
\label{defilt.2a} Use the notation of \ref{defilt.2}, which could help the reader to verify the following.
\begin{enumerate}
 \item If $a_1 ,\ldots, a_n$ form a $\fb$-filter-regular sequence on $M$, then $a_1^{\alpha_1} ,\ldots, a_n^{\alpha_n}$ is a $\fb$-filter-regular sequence on $M$ for all choices of $\alpha_1, \ldots, \alpha_n \in \N$.
\item If $(\Lambda, \leq)$ is a
directed partially or\-dered set, and
$(W_{\alpha})_{\alpha \in \Lambda}$ is a direct system of
$A$-modules over $\Lambda$ with direct limit $W_{\infty}$, and if $a_1 ,\ldots, a_n$ is a $\fb$-filter-regular sequence on $W_{\alpha}$ for each $\alpha \in \Lambda$, then the $a_1 ,\ldots, a_n$ form a $\fb$-filter-regular sequence on $W_{\infty}$.
\end{enumerate}
\end{rmk}

A straightforward prime-avoidance argument will prove the following lemma.

\begin{lem}
\label{defilt.3} {\rm (See Khashyarmanesh--Salarian--Zakeri \cite[p.\ 39]{KhaSalZak98}.)} Assume that $M$ is non-zero and finitely generated. Suppose that $a_1 ,\ldots, a_n$ is a $\fb$-filter-regular sequence on $M$ composed of elements of $\fb$.  Then there exists $a_{n+1} \in \fb$ such that
$a_1 ,\ldots, a_n,a_{n+1}$ is a $\fb$-filter-regular sequence on $M$.

It follows that there exists an infinite $\fb$-filter-regular sequence on $M$ composed of elements of $\fb$.
\end{lem}

A triangular subset $U$ of $A^n$ is {\em expanded\/} if, whenever $(a_1, \ldots, a_n) \in U$, then $(a_1, \ldots, a_i,1,\ldots,1) \in U$ for all $i = 0, \ldots,n-1$. When $n > t \geq 1$, the {\em restriction of $U$ to $A^t$\/} is
\[
\{ (b_1, \ldots, b_t) \in A^t : \mbox{there exist $b_{t+1}, \ldots, b_n \in A$ such that $(b_1, \ldots, b_t,b_{t+1}, \ldots, b_n) \in U$} \}.
\]

\begin{ntn}
\label{defilt.4} Let $\mathcal{V} = (V_i)_{i\in\N}$ be a chain of expanded triangular subsets on $A$ in the sense of \cite[p.\ 420]{O'Car83}. Thus
$V_i$ is an expanded triangular subset of $A^i$ for all $i \in \N$, and $V_i$ is the restriction of $V_{i+1}$ to $A^i$ for all $i \in \N$.

We can form the associated complex of modules of generalized fractions
\[
0 \lra M \stackrel{d^0}{\lra}V_1^{-1}M \stackrel{d^1}{\lra} \cdots \lra V_i^{-i}M \stackrel{d^i}{\lra} V_{i+1}^{-(i+1)}M \lra \cdots,
\]
which we denote by $C\big({\mathcal V},M\big)$. Here, $d^0(m) = m/(1)$ for all $m \in M$, and, for $i \in \N$,
\[
d^i \Big(\frac{m}{(v_1, \ldots, v_i)}\Big) = \frac{m}{(v_1, \ldots, v_i,1)} \quad \mbox{for all~} m \in M,(v_1, \ldots, v_i)\in V_i.
\]
\end{ntn}

Part (i) of the next proposition is an easy consequence of \cite[Proposition (2.1)]{42} and \ref{defilt.2}.

\begin{thm}\label{defilt.5} {\rm (See Khashyarmanesh--Salarian--Zakeri \cite[Theorems 1.1, 1.2]{KhaSalZak98}.)}
 Let the situation and notation be as in\/ {\rm \ref{defilt.4}}.
 \begin{enumerate}
 \item Each member of each $V_i~(i\in\N)$ is a $\fb$-filter-regular sequence on $M$ if and only if all the cohomology modules $H^i(C({\mathcal V},M))~(i\in\nn)$ of the complex $C({\mathcal V},M)$ are $\fb$-torsion.
\item Assume that the conditions in\/ {\rm (i)} are satisfied, and also that $H^j_{\fb}(V_i^{-i}M) = 0$ for all $j\in \nn$ and $i \in \N$. Then
$H^i(C({\mathcal V},M)) \cong H^i_{\fb}(M)$ for all $i\in\nn$.
\end{enumerate}
\end{thm}

\begin{proof}[Proof.] (i) By \ref{defilt.2}, each member of each $V_i$ is a $\fb$-filter-regular sequence on $M$ if and only if, for each $i \in \N$, for each $\fp \in \Spec(A) \setminus \Var(\fb)$, and for each $(a_1, \ldots, a_i) \in V_i$, the sequence $a_1/1 ,\ldots, a_i/1$ of natural images in $A_{\fp}$ is a poor $M_{\fp}$-sequence. By \cite[(2.1)]{42}, this is the case if and only if all the cohomology modules of the complex $C({\mathcal V},M)$ have support contained in $\Var(\fb)$, that is, are $\fb$-torsion.

(ii) This is Theorem 1.2 of Khashyarmanesh--Salarian--Zakeri \cite{KhaSalZak98}. Their proof is a statement that one can use the arguments in the proof of \cite[Theorem 2.4]{42}. In that theorem, the underlying ring is local and the module $M$ is finitely generated. The argument there works in our more general situation here. We point out that Khashyarmanesh, Salarian and Zakeri do not assume that $M$ is finitely generated in their Theorems 1.1 and 1.2 of \cite{KhaSalZak98}.
\end{proof}

\begin{note} It is important to note that $M$ is not assumed to be finitely generated in \ref{defilt.5}.  In the corollary below, we generalize \cite[Consequences 1.3(i)]{KhaSalZak98} to the case of an arbitrary, not necessarily finitely generated, $A$-module $M$.
\end{note}

\begin{cor}
\label{defilt.6} Let $\mathbf{f} = (f_1, \ldots, f_t, \ldots)$ be an infinite $\fb$-filter-regular sequence on $M$ composed of elements of $\fb$.  For each $i \in \N$, set $U_i := U_{(f_1, \ldots, f_i)} = \{ (f_1^{\alpha_1}, \ldots, f_i^{\alpha_i}) : \alpha_1, \ldots, \alpha_i \in \nn\},$ an expanded triangular subset of $A^i$. Also set ${\mathcal U}_{\mathbf{f}} = (U_i)_{i\in\N}$, a chain of expanded triangular subsets on $A$. We can form the complex $C({\mathcal U}_{\mathbf{f}},M)$.

There are $A$-isomorphisms $H^i(C({\mathcal U}_{\mathbf{f}},M)) \cong H^i_{\fb}(M)$ for all $i \in \nn$.
\end{cor}

\begin{note}
Of course, for $i$ sufficiently large, for example for $i$ greater than the arithmetic rank of $\fb$, the local cohomology module $H^i_{\fb}(M)$ will be zero.
\end{note}

\begin{proof}[Proof.] Let $i\in\N$, and set $V_i := \{ (f_1^{\alpha_1}, \ldots, f_i^{\alpha_i}) : \alpha_1, \ldots, \alpha_i \in \N\}$, a triangular subset of $A^i$. Now $V_i \subset U_i$ and the natural homomorphism $V_i^{-i}M \lra U_i^{-i}M$ is an isomorphism. Therefore we have $H^j_{\fb}(U_i^{-i}M) \cong H^j_{\fb}(V_i^{-i}M)$ for all $j\in \nn$, and this is zero by \cite[2.2]{34}. Every member of $U_i$ (including those with some components equal to $1$) is a $\fb$-filter-regular sequence on $M$ (by \ref{defilt.2a}(i)). Therefore, by the Theorem \ref{defilt.5} of Khashyarmanesh--Salarian--Zakeri, there are $A$-isomorphisms $H^i(C({\mathcal U}_{\mathbf{f}},M)) \cong H^i_{\fb}(M)$ for all $i\in \nn$.
\end{proof}

Our remaining results in this section concern modules over $R$: see \ref{GLF.1}.

\begin{thm}
\label{alg.3} Over $R$, let ${\mathcal M}$ be an $F$-finite $F$-module with root $N$. By\/ {\rm \ref{defilt.3}}, there exists an infinite $\fa$-filter-regular sequence $\mathbf{g} := (g_i)_{i\in\N}$ on $N$ composed of elements of $\fa$. Then $(g_i)_{i\in\N}$ is an $\fa$-filter-regular sequence on ${\mathcal M}$.
\end{thm}

\begin{proof}[Proof.] Let $\fp \in \Spec(R) \setminus \Var(\fa)$. By \ref{defilt.2}, the sequence $g_1/1 ,\ldots, g_h/1, \ldots$ of natural images in $R_{\fp}$ is a poor $N_{\fp}$-sequence. Therefore $g_1^p/1 ,\ldots, g_h^p/1, \ldots$ is a poor $F'(N_{\fp})$-sequence, where $F'$ denotes the (exact) Frobenius functor on the category of modules over the regular local ring $R_{\fp}$. It follows from this that $g_1/1 ,\ldots, g_h/1,\ldots$ is a poor $F'(N_{\fp})$-sequence: see \cite[Exercise 12(c), p.\ 103]{IK}. We can now use \ref{alg.13} to show that the sequence $g_1/1 ,\ldots, g_h/1,\ldots$ of natural images in $R_{\fp}$ is a poor $(F(N))_{\fp}$-sequence. This is true for all $\fp \in \Spec(R) \setminus \Var(\fa)$. Therefore, by \ref{defilt.2} again, $(g_i)_{i\in\N}$ is an $\fa$-filter-regular sequence on $F(N)$. It follows that $(g_i)_{i\in\N}$ is an $\fa$-filter-regular sequence on $F^j(N)$ for all $j \in \N$. Since ${\mathcal M}$ is the direct limit of a direct system with constituent $R$-modules $(F^i(N))_{i\in\nn}$, it follows from \ref{defilt.2a}(ii) that  $(g_i)_{i\in\N}$ is an $\fa$-filter-regular sequence on ${\mathcal M}$.
\end{proof}

\begin{strat}
\label{alg.10} Here we set out a strategy that can help with the analysis of an $F$-finite $F$-module ${\mathcal M}$ over $R$.
\begin{enumerate}
\item Let $N$ be a root for ${\mathcal M}$. Necessarily, $N$ is finitely generated.
\item By\/ {\rm \ref{defilt.3}}, there exists an infinite $\fa$-filter-regular sequence $\mathbf{g} := (g_i)_{i\in\N}$ on $N$ composed of elements of $\fa$.
\item By \ref{alg.3}, $(g_i)_{i\in\N}$ is an $\fa$-filter-regular sequence on ${\mathcal M}$.
\item For each $i \in\N$, set $\mathbf{g}_i := (g_1, \ldots, g_i)$ and consider $R_{\mathbf{g}_i}$, the module of generalized fractions of $R$ with respect to
$
U_{\mathbf{g}_i} := \{(g_1^{\beta_1}, \ldots, g_i^{\beta_i}) : \beta_1, \ldots, \beta_i \in \nn\}.
$
Now ${\mathcal U}_{\mathbf{g}} := (U_{\mathbf{g}_i})_{i\in\N}$ is a chain of triangular sets on $R$ and
we can form the complex of generalized fractions  $C({\mathcal U}_{\mathbf{g}},R)$, namely
\[
0 \stackrel{d^{-1}}{\lra} R \stackrel{d^0}{\lra} R_{\mathbf{g}_1}\stackrel{d^1}{\lra} \cdots \lra R_{\mathbf{g}_{h-1}} \stackrel{d^{h-1}}{\lra} R_{\mathbf{g}_h} \stackrel{d^h}{\lra} \cdots,
\]
as in \ref{defilt.6}.
\item We can also form the complex of generalized fractions $C({\mathcal U}_{\mathbf{g}},{\mathcal M})$, again as in \ref{defilt.6}; this has the form
\[
0 \stackrel{e^{-1}}{\lra} {\mathcal M} \stackrel{e^0}{\lra} {\mathcal M}_{\mathbf{g}_1}\stackrel{e^1}{\lra} \cdots \lra {\mathcal M}_{\mathbf{g}_{h-1}} \stackrel{e^{h-1}}{\lra} {\mathcal M}_{\mathbf{g}_h} \stackrel{e^h}{\lra} \cdots \mbox{;}
\] note that, by\/ {\rm \ref{calc.6}}, this is isomorphic to $C({\mathcal U}_{\mathbf{g}},R)\otimes_R{\mathcal M}$.
\item It now follows from\/ {\rm (\ref{defilt.6})} that
    $H^i(C({\mathcal U}_{\mathbf{g}},{\mathcal M})) \cong H^i_{\fa}({\mathcal M})$ for all $i \in\nn$.
    \end{enumerate}
\end{strat}

We plan to use the above strategy in our discussion of Lyubeznik numbers in the next section. However, the strategy recovers some results of Lyubeznik.

\begin{thm}
\label{alg.4} {\rm (Lyubeznik \cite[Proposition 2.10]{Lyube97}.)} Let ${\mathcal M}$ be an $F$-finite $F$-module over $R$. Then the local cohomology module $H^i_{\fa}({\mathcal M})$ is an $F$-finite $F$-module, for all $i \in \nn$.
\end{thm}

\begin{proof} Use Strategy \ref{alg.10}, and the notation thereof. The conclusion is that
    $H^i(C({\mathcal U}_{\mathbf{g}},{\mathcal M})) \cong H^i_{\fa}({\mathcal M})$ for all $i \in\nn$.

Let $\nu : \mathcal{M} \stackrel{\cong}{\lra} F(\mathcal{M}) = R'\otimes_R\mathcal{M}$ be the structural isomorphism. It follows from \ref{1006}(i) that ${\mathcal M}_{\mathbf f}$ is an $F$-module for all ${\mathbf f} = (f_1, \ldots, f_n) \in R^n$; the details given in the statement of \ref{mm.1g} show that, for each $i \in \N$, a structural isomorphism for ${\mathcal M}_{\mathbf g_i}$ is
$\theta_i : \mathcal{M}_{{\mathbf g}_i} \stackrel{\cong}{\lra} F(\mathcal{M}_{{\mathbf g}_i})$, where
\[
\theta_i\left( \frac{\nu^{-1}\left(\sum_{k=1}^w r_k' \otimes m_k\right)}{(g_1^{j_1}, \ldots,g_i^{j_i})} \right) = \sum_{k=1}^w g_1^{j_1(p-1)}\ldots g_i^{j_i(p-1)}r_k' \otimes \frac{m_k}{(g_1^{j_1}, \ldots,g_i^{j_n})}
\]
for all $j_1, \ldots, j_i \in \nn$, $w \in \N$, $r_1', \ldots, r'_w \in R$ and $m_1, \ldots, m_w \in \mathcal{M}$. It is routine to check that $\theta_1 \circ e^0 = F(e^0) \circ \nu$ and $\theta_{i+1} \circ e^i = F(e^i) \circ \theta_i$ for all $i\in \N$. Hence $C({\mathcal U}_{\mathbf{g}},{\mathcal M})$ is actually a complex in the category ${\mathcal F}$ of $F$-modules and $F$-homomorphisms. By \cite[1.1]{Lyube97}, this category is Abelian, and, for an $h \in \nn$, both ${\mathcal K} := \Ker e_h$ and ${\mathcal L} := \Ima e_{h-1}$ are $F$-submodules of $\mathcal{M}_{{\mathbf g}_h}$. Since $\mathcal{M}_{{\mathbf g}_h}$ is an $F$-finite $F$-module (by \ref{add.5}), it then follows from
\cite[Proposition 2.5(b)]{Lyube97} that both ${\mathcal K}$ and ${\mathcal L}$ are $F$-finite $F$-modules. Finally, it follows from the discussion of the full Abelian subcategory of ${\mathcal F}$ formed by the $F$-finite $F$-modules in \cite[Theorem 2.8]{Lyube97} (and its proof) that ${\mathcal K}/{\mathcal L} = H^h(C({\mathcal U}_{\mathbf{g}},{\mathcal M}))$ is an $F$-finite $F$-module.  The isomorphism
$H^h(C({\mathcal U}_{\mathbf{g}},{\mathcal M})) \cong H^h_{\fa}({\mathcal M})$ enables us to conclude that $H^h_{\fa}({\mathcal M})$ is an $F$-finite $F$-module.
\end{proof}

\begin{cor}
\label{alg.9} {\rm (Lyubeznik \cite{Lyube97}.)} Let ${\mathcal M}$ be an $F$-finite $F$-module over $R$, let $\fa_1, \ldots \fa_d$ be ideals of $R$ and let $i_1, \ldots, i_d\in \nn$. Then the local cohomology module
$$H^{i_d}_{\fa_d}(H^{i_{d-1}}_{\fa_{d-1}}( \cdots (H^{i_1}_{\fa_1}({\mathcal M}))\cdots))$$
is an $F$-finite $F$-module.
\end{cor}

\section{Lyubeznik numbers}
\label{Lnos}

We plan to show, in this section, that the ideas we have presented so far in this paper have application to the calculation of Lyubeznik numbers. We recalled the definition of Lyubeznik numbers at the end of the Introduction. This section provides our algorithm for the calculation of the Lyubeznik numbers of certain localizations of affine algebras over fields of prime characteristic.

\begin{ntn}
\label{Lnos.99} Throughout this section, let ${\mathbb K}$ denote a field of prime characteristic $p$, and let $R$ denote the polynomial ring ${\mathbb K}[X_1, \ldots, X_n]$ over ${\mathbb K}$ in the indeterminates $X_1, \ldots, X_n$. We are interested in the affine ${\mathbb K}$-algebra $R/\fc$, where $\fc$ is a proper ideal of $R$, and in finding the Lyubeznik numbers of the local ring $A := (R/\fc)_{\fm/\fc}$, where $\fm$ is a maximal ideal of $R$ containing $\fc$. Note that $A\cong R_{\fm}/\fc R_{\fm}$.

We shall often be interested in the case where $\fm$ is the ideal of $R$ generated by all the indeterminates $X_1, \ldots, X_n$. Recall that, then, the completion of $R_{\fm}$ is ${\mathbb K}[[X_1, \ldots, X_n]]$. It is an easy consequence of the Flat Base Change Theorem for local cohomology  that the Lyubeznik table of $A$ is equal to the Lyubeznik table of the completion $\widehat{A} \cong {\mathbb K}[[X_1, \ldots, X_n]]/\fc {\mathbb K}[[X_1, \ldots, X_n]]$. Throughout, $i$ and $j$ will denote arbitrary non-negative integers.
\end{ntn}

We have $\lambda_{i,j}(A) = \mu^i(\fm R_{\fm}, H^{n-j}_{\fc R_{\fm}}(R_{\fm}))$. A route to calculation of these Bass numbers is provided by Huneke--Sharp \cite[Corollary 3.7]{58}, from which it follows that $H^i_{\fm R_{\fm}}(H^{n-j}_{\fc R_{\fm}}(R_{\fm}))$ is an injective $R_{\fm}$-module. The following lemma will be helpful.

\begin{lem}\label{Lnos.97} Let ${\mathcal N}$ be an $\fm$-torsion $R$-module whose localization ${\mathcal N}_{\fm}$ at the maximal ideal $\fm$ is an injective $R_{\fm}$-module. Then ${\mathcal N}$ is an injective $R$-module isomorphic to $E(R/\fm)^{\mu^0(\fm,{\mathcal N})}$.
\end{lem}

\begin{proof}[Proof.] Suppose there exists $\fp \in \Spec(R)$  with $\fp \neq \fm$ such that $\mu^0(\fp, {\mathcal N}) > 0$. Because $\fm$ is maximal, there exists $r \in \fm \setminus \fp$. Also, there exists $0 \neq x \in {\mathcal N}$ such that $(0:x) = \fp$. As ${\mathcal N}$ is $\fm$-torsion, there exists $t \in \N$ such that $r^tx = 0$. Therefore $r^t \in \fp$, and we have a contradiction. It follows that $E^0({\mathcal N})$, the injective envelope of ${\mathcal N}$, is a direct sum of copies of $E(R/\fm)$, and so is $\fm$-torsion. So also is $E^0({\mathcal N})/{\mathcal N}$. An easy inductive argument now shows that $E^i({\mathcal N})$ is a direct sum of copies of $E(R/\fm)$, that is, $E^i({\mathcal N}) \cong E(R/\fm)^{\mu^i(\fm,{\mathcal N})} $, for all $i \in \nn$.

But $\mu^i(\fm,{\mathcal N}) = \mu^i(\fm R_{\fm},{\mathcal N}_{\fm})$ and this is $0$ when $i > 0$ because ${\mathcal N}_{\fm}$ is an injective $R_{\fm}$-module. Therefore $E^i({\mathcal N}) = 0$ for all $i \in \N$, so that ${\mathcal N}$ is an injective $R$-module isomorphic to $E^0({\mathcal N})$.
\end{proof}

\begin{rmk}
\label{Lnos.98} We return to the calculation of $\lambda_{i,j}(R_{\fm}/\fc R_{\fm}) = \mu^i(\fm R_{\fm}, H^{n-j}_{\fc R_{\fm}}(R_{\fm}))$. It follows from Huneke--Sharp \cite[Corollary 3.7]{58} that $H^i_{\fm R_{\fm}}(H^{n-j}_{\fc R_{\fm}}(R_{\fm}))$ is an injective $R_{\fm}$-module isomorphic to the direct sum of $\mu^i(\fm R_{\fm}, H^{n-j}_{\fc R_{\fm}}(R_{\fm}))$ copies of $E_{R_{\fm}}(R_{\fm}/\fm R_{\fm})$ (and $\mu^i(\fm R_{\fm}, H^{n-j}_{\fc R_{\fm}}(R_{\fm})) = \mu^i(\fm, H^{n-j}_{\fc}(R))$ is finite). Now $H^i_{\fm R_{\fm}}(H^{n-j}_{\fc R_{\fm}}(R_{\fm})) \cong (H^i_{\fm}(H^{n-j}_{\fc}(R)))_{\fm}$; it follows from \ref{Lnos.97} applied to the $\fm$-torsion $R$-module $H^i_{\fm}(H^{n-j}_{\fc}(R))$ that $H^i_{\fm}(H^{n-j}_{\fc}(R))$ is injective and its $0$th Bass number with respect to $\fm$ is $\mu^i(\fm, H^{n-j}_{\fc}(R))$.

   Consequently, $\lambda_{i,j}(A)$ is equal to the dimension, as a vector space over $R/\fm$, of the annihilator
$(0:_{H^i_{\fm}(H^{n-j}_{\fc}(R))}\fm)$.
\end{rmk}

In Discussion \ref{Lnos.2} below, we describe a sequence of steps which provide the basis for our algorithm. For full details of the Macaulay2 coding, the reader is referred to \[ \mbox{http://www.katzman.staff.shef.ac.uk/LyubeznikNumbers/} \] We mention now that, by the image and cokernel of an $n \times t$ matrix $Q$ over $R$, we mean the image and cokernel of the $R$-homomorphism $\nu_Q : R^t \lra R^n$ given by left multiplication by $Q$, that is $\nu_Q((r_1, \ldots,r_t)) = (Q(r_1, \ldots,r_t)^T)^T$ for all $(r_1, \ldots,r_t) \in R^t$, where $^T$ denotes matrix transpose.

\begin{disc}
\label{Lnos.2} Let the situation and notation be as in \ref{Lnos.99}, and use ${\mathcal M}$ to denote the $F$-finite $F$-module $H^{n-j}_{\fc}(R)$ (see \ref{alg.4}). This discussion is a recipe for finding the Lyubeznik number $\lambda_{i,j}(R_{\fm}/\fc R_{\fm})$. We have seen in \ref{Lnos.98} that this is equal to the $R/\fm$-vector-space dimension of $\left( 0 :_{H^i_{\fm}({\mathcal M})} \fm\right).$
\begin{enumerate}
\item First find a generating morphism $\alpha : L \lra F(L)$ for ${\mathcal M} = H^{n-j}_{\fc}(R)$. One (but not the only) way of doing this is to follow the recipe of Lyubeznik \cite[Proposition 1.11(a)]{Lyube97}: the $R$-module homomorphism
    \[
    \alpha : L := \Ext^{n-j}_R(R/\fc, R) \lra F(\Ext^{n-j}_R(R/\fc, R)) = \Ext^{n-j}_R(F(R/\fc), F(R))
    \]
    induced by the maps $\psi : F(R/\fc) \lra R/\fc$ and $\phi : R \lra F(R)$, for which $\psi (r'\otimes (r + \fc)) = r^pr' + \fc$ and $\phi(r) = r \otimes 1$ for all $r\in R$ and $r' \in R'$, is a generating morphism for $H^{n-j}_{\fc}(R)$.
\item The next step is to find a root morphism $\theta : N \lra F(N)$ for ${\mathcal M}$, given the generating morphism $\alpha : L \lra F(L)$ of part (i) for it. An effective method for doing this is provided by Lyubeznik \cite[Proposition 2.3]{Lyube97}: for each $i \in \N$, let $\alpha_i : L \lra F^i(L)$ be the composition
    \[
\begin{picture}(300,15)(-150,35)
\put(-120,40){\makebox(0,0){$                   L
$}}
\put(-60,40){\makebox(0,0){$                    F(L)
$}}
\put(-7,40){\makebox(0,0){$                      \cdots
$}}
\put(52,40){\makebox(0,0){$                     F^{i-1}(L)
$}}
\put(120,40){\makebox(0,0){$                    F^{i}(L)
$}}
\put(-110,40){\vector(1,0){36}}
\put(-90,44){\makebox(0,0){$^{             \alpha
}$}}
\put(-47,40){\vector(1,0){29}}
\put(-34,44){\makebox(0,0){$^{             F(\alpha)
}$}}
\put(0,40){\vector(1,0){30}}
\put(71,40){\vector(1,0){32}}
\put(87,44){\makebox(0,0){$^{             F^{i-1}(\alpha)
}$}}
\end{picture}
\]
(interpret $F^0(\alpha)$ as $\alpha$); let $k$ be the smallest $i$ for which $\Ker \alpha_i = \Ker \alpha_{i+1}$ (and such exists because $L$ is finitely generated); Lyubeznik \cite[Proposition 2.3]{Lyube97} shows that $\Ima \alpha_k =: N$ is a root for $H^{n-j}_{\fc}(R)$, and that the restriction of $F^k(\alpha)$ to $N$  yields a root morphism for $H^{n-j}_{\fc}(R)$.
\item Now use Strategy \ref{alg.10} (and the notation thereof) on ${\mathcal M} = H^{n-j}_{\fc}(R)$, but with the choice $\fa = \fm$. We find an infinite $\fm$-filter-regular sequence $\mathbf{g} := (g_i)_{i\in\N}$, composed of elements of $\fm$, on $N$; by \ref{alg.3}, that is automatically an $\fm$-filter-regular sequence on ${\mathcal M}$; we also get the chain of triangular sets ${\mathcal U}_{\mathbf{g}} := (U_{\mathbf{g}_i})_{i\in\N}$ on $R$, the complexes of generalized fractions  $C({\mathcal U}_{\mathbf{g}},{\mathcal M})$ and $C({\mathcal U}_{\mathbf{g}},R)$, and the isomorphisms
    $H^i(C({\mathcal U}_{\mathbf{g}},{\mathcal M})) \cong H^i_{\fm}({\mathcal M})$ for $i \in\nn$. Although theoretically this step deals with an infinite $\fm$-filter-regular sequence $\mathbf{g} := (g_i)_{i\in\N}$ on $N$ and ${\mathcal M}$, in practice one does not need details of the terms of the sequence beyond a certain stage: if one is interested in, say, $H^k(C({\mathcal U}_{\mathbf{g}},{\mathcal M}))$, then one only needs details of $g_1, \ldots, g_{k+1}$.

\item Let us take stock. Recall that we wish to calculate the Lyubeznik number $\lambda_{i,j}(R_{\fm}/\fc R_{\fm})$. We have seen in \ref{Lnos.98} that this is equal to the $R/\fm$-vector-space dimension of $\left( 0 :_{H^i_{\fm}({\mathcal M})} \fm\right).$ Since $H^i_{\fm}({\mathcal M})$ is $\fm$-torsion and $\fm$ is maximal, this dimension is the $0$th Bass number $\mu^0(\fm, H^i_{\fm}({\mathcal M})).$ We plan to use the isomorphism $H^i(C({\mathcal U}_{\mathbf{g}},{\mathcal M})) \cong H^i_{\fm}({\mathcal M})$ and a root $J$ for $H^i(C({\mathcal U}_{\mathbf{g}},{\mathcal M}))$ to complete the calculation, for it will then follow from \ref{alg.5}(ii) that
    \begin{align*}
    \lambda_{i,j}(R_{\fm}/\fc R_{\fm}) &= \dim_{R/\fm}\left( 0 :_{H^i_{\fm}({\mathcal M})} \fm\right)= \mu^0(\fm, H^i_{\fm}({\mathcal M}))= \mu^0(\fm, H^i(C({\mathcal U}_{\mathbf{g}},{\mathcal M})))\\& = \mu^0(\fm, J)= \dim_{R/\fm} (0:_J\fm). \end{align*}
    We therefore direct our attention to finding a root $J$ for $H^i(C({\mathcal U}_{\mathbf{g}},{\mathcal M}))$ and the $R/\fm$-vector-space dimension of the annihilator $(0:_J\fm)$.

\item    Each $R_{\mathbf{g}_h} = U_{\mathbf{g}_h}^{-h}R$ is the union of its cyclic submodules $\big(R\big(1/\mathbf{g}_h^{p^e}\big)\big)_{e\in \nn}$.
Now ${\mathcal M}$ is the direct limit of the direct system associated to the sequence
\[
\begin{picture}(300,15)(-150,35)
\put(-120,40){\makebox(0,0){$                   N
$}}
\put(-60,40){\makebox(0,0){$                    F(N)
$}}
\put(-7,40){\makebox(0,0){$                      \cdots
$}}
\put(52,40){\makebox(0,0){$                     F^i(N)
$}}
\put(120,40){\makebox(0,0){$                    F^{i+1}(N)
$}}
\put(-110,40){\vector(1,0){36}}
\put(-90,44){\makebox(0,0){$^{             \theta
}$}}
\put(-47,40){\vector(1,0){29}}
\put(-34,44){\makebox(0,0){$^{             F(\theta)
}$}}
\put(0,40){\vector(1,0){36}}
\put(20,44){\makebox(0,0){$^{             F^{i-1}(\theta)
}$}}
\put(67,40){\vector(1,0){32}}
\put(83,44){\makebox(0,0){$^{             F^i(\theta)
}$}}
\put(142,40){\vector(1,0){28}}
\put(180,40){\makebox(0,0){$                      \cdots
$.}}
\end{picture}
\]
It therefore follows from\/ {\rm \ref{add.3}} that $R_{(g_1,\ldots,g_h)} \otimes_R{\mathcal M}$ is the direct limit of the direct system associated to the diagonal sequence
\[
{\textstyle 0 \lra R\left(\frac{1}{\mathbf{g}_h}\right) \otimes_RN\lra R\left(\frac{1}{\mathbf{g}^p_h}\right)\otimes_R F(N)\lra \cdots \lra R\Big(\frac{1}{\mathbf{g}_h^{p^e}}\Big)\otimes_R F^e(N) \lra\cdots}.
\]
These sequences, for various $h$s, fit as the columns in the commutative diagram
\begin{equation*}%\label{CD1}
\xymatrix{
\dots \ar[r]^<{\quad \quad \quad \quad d^{i-2}\lceil\otimes \Id}  & R\big(\frac{1}{\mathbf{g}_{i-1}}\big)\otimes N \ar[r]^{d^{i-1}\lceil\otimes \Id} \ar[d]^{\iota \otimes \theta}   & R\big(\frac{1}{\mathbf{g}_{i}}\big)\otimes N \ar[r]^{d^{i}\lceil\otimes \Id}   \ar[d]^{\iota \otimes \theta} & R\big(\frac{1}{\mathbf{g}_{i+1}}\big)\otimes N \ar[r]^>{d^{i+1}\lceil\otimes \Id \quad \quad \quad \quad } \ar[d]^{\iota \otimes \theta} & \dots\\
\dots \ar[r]^<{\quad \quad d^{i-2}\lceil\otimes \Id}  & R\big(\frac{1}{\mathbf{g}_{i-1}^p}\big)\otimes F(N) \ \ \  \ar[r]^{d^{i-1}\lceil\otimes \Id} \ar[d]^{\iota \otimes F(\theta)}   & R\big(\frac{1}{\mathbf{g}_{i}^p}\big)\otimes F(N) \ar[r]^{d^{i}\lceil\otimes \Id}  \ar[d]^{\iota \otimes F(\theta)}  & R\big(\frac{1}{\mathbf{g}_{i+1}^p}\big)\otimes F(N) \ar[r]^>{d^{i+1}\lceil\otimes \Id \quad \quad \quad \quad }  \ar[d]^{\iota \otimes F(\theta)} & \dots\\
& \vdots & \vdots & \vdots & \quad \quad \quad ,\\
}
\end{equation*}
in which `$\iota$' is  used to indicate an appropriate inclusion map, `$\Id$' is used to denote an appropriate identity map, the symbol `$\lceil$' is used to denote restriction (of maps) and all the tensor products are over $R$.
The rows of this commutative diagram form a direct system (over $\N$) of complexes of $R$-modules and $R$-homomorphisms, and of chain maps of such complexes. By \ref{add.3}, the direct limit of these complexes is isomorphic to the complex
$C({\mathcal U}_{\mathbf{g}},R)\otimes_R{\mathcal M}$. By \ref{calc.6}, this is isomorphic to $C({\mathcal U}_{\mathbf{g}},{\mathcal M})$. We are aiming to find a root for $H^i(C({\mathcal U}_{\mathbf{g}},{\mathcal M}))$.

Now, for a direct system of complexes and chain maps of complexes, the operation of taking direct limits commutes with the operation of taking cohomology. It follows that a generating morphism for the ($F$-finite) $F$-module $H^i(C({\mathcal U}_{\mathbf{g}},{\mathcal M}))$ is the map
    \[
    \Ker (d^i\lceil \otimes \Id_N) / \Ima (d^{i-1}\lceil \otimes \Id_N) \lra \Ker (d^i\lceil \otimes \Id_{F(N)}) / \Ima (d^{i-1}\lceil \otimes \Id_{F(N)})
    \]
    induced by $\iota \otimes \theta$. (Do not forget that $F(R(1/\mathbf{g}_i)) \cong R(1/\mathbf{g}_i^p)$: see \ref{gf.107}(i).)

The annihilator of the generalized fraction $1/\mathbf{g}_i = 1/(g_1, \ldots, g_i)$ in $R_{\mathbf {g}_i}$ is $(g_1, \ldots, g_i)^{\lowlim}$, by \ref{gf.103}.  The map $d^i \lceil : R(1/\mathbf{g}_i) \lra R(1/\mathbf{g}_{i+1})$ is isomorphic to $$R/(g_1, \ldots, g_i)^{\lowlim} \stackrel{g_{i+1}}{\lra} R/(g_1, \ldots, g_i,g_{i+1})^{\lowlim}.$$ (The notation means that the homomorphism is induced by multiplication by $g_{i+1}$.)
Also, because $(g_1\ldots g_{i})^{p-1}(1/(g_1^p, \ldots, g_i^p)) = 1/(g_1, \ldots, g_i),$ the inclusion map $\iota : R(1/\mathbf{g}_i) \lra R(1/\mathbf{g}_i^p)$ is isomorphic to $$R/(g_1, \ldots, g_i)^{\lowlim} \stackrel{\pi_i}{\lra} R/(g_1^p, \ldots, g_i^p)^{\lowlim},$$ where $\pi_j = g_1^{p-1}\ldots g_j^{p-1}$ for all $j \in \N$. The concept of the {\em expansion\/} of a triangular subset of $R^n$ (\cite[p.\ 38]{32}) enables one to see quickly that the annihilator of $1/(g_1^p, \dots,g_i^p)$ is the same whether we consider this generalized fraction as a member of $R_{\mathbf{g}_i}$ or $R_{\mathbf{g}_i^p}$.

\item For $j \in\N$, set $\fc_j := (g_1, \ldots, g_j)^{\lowlim}$; calculate $\fc_i$ and $\fc_{i+1}$, with the aid of \ref{gf.105}. Express $N$ as $R^n/\Ima K = \Coker K$ for a suitable $n \times t$ matrix $K$ over $R$; we shall denote by $K^{[p]}$ the matrix obtained from $K$ by raising all its entries to the $p$th power. Then
    \begin{align*}
    R(1/\mathbf{g}_i) \otimes_RN & \cong (R/\fc_i) \otimes_R (R^n/\Ima K)\\ &\cong (R^n/\Ima K)/\fc_i (R^n/\Ima K) \cong R^n/(\fc_i R^n+ \Ima K).
    \end{align*}
The above comments enable us to see that the commutative diagram in (v) is isomorphic to
\begin{equation*}%\label{CD2}
\xymatrix{
\dots \ar[r]^<{\quad \quad \quad \quad g_{i-1}\otimes \Id}  & \frac{R}{\fc_{i-1}}\otimes N \ar[r]^{g_{i}\otimes \Id} \ar[d]^{\pi_{i-1} \otimes \theta}   & \frac{R}{\fc_{i}}\otimes N \ar[r]^{g_{i+1}\otimes \Id }   \ar[d]^{\pi_i \otimes \theta} & \frac{R}{\fc_{i+1}}\otimes N \ar[r]^>{g_{i+2}\otimes \Id \quad \quad \quad } \ar[d]^{\pi_{i+1} \otimes \theta} & \dots\\
\dots \ar[r]^<{\quad \quad g_{i-1}^p\otimes \Id}  & \frac{R}{\fc_{i-1}^{[p]}}\otimes F(N) \ \ \  \ar[r]^{g_i^p\otimes \Id} \ar[d]^{\pi_{i-1}^p \otimes F(\theta)}   & \frac{R}{\fc_{i}^{[p]}}\otimes F(N) \ar[r]^{g_{i+1}^p\otimes \Id}  \ar[d]^{\pi_{i}^p \otimes F(\theta)}  & \frac{R}{\fc_{i+1}^{[p]}}\otimes F(N) \ar[r]^>{g_{i+2}^p\otimes \Id\quad \quad }  \ar[d]^{\pi_{i+1}^p \otimes F(\theta)} & \dots\\
& \vdots & \vdots & \vdots & ,\\
}
\end{equation*}
where all the tensor products are over $R$.
\item Next, find an $n \times n$ matrix $U$ over $R$ such that the map $\theta : N \lra F(N)$ is isomorphic to $R^n/\Ima K \stackrel{U}{\lra} R^n/\Ima K^{[p]}$, the map being induced by multiplication on the left by $U$.
It then follows that the commutative diagrams in (v) and (vi) are isomorphic to
\begin{equation*}%\label{CD2}
\xymatrix{
\dots \ar[r]^<{\quad \quad \quad g_{i-1}}  & R^n/(\fc_{i-1} R^n+ \Ima K) \ar[r]^{g_{i}} \ar[d]^{\pi_{i-1}U}   & R^n/(\fc_{i} R^n+ \Ima K) \ar[r]^{g_{i+1}}   \ar[d]^{\pi_i U} & R^n/(\fc_{i+1} R^n+ \Ima K) \ar[r]^>{g_{i+2}\quad \quad \quad \quad } \ar[d]^{\pi_{i+1}U} & \dots\\
\dots \ar[r]^<{\quad \quad \quad  g_{i-1}^p}  & R^n/(\fc_{i-1}^{[p]} R^n+ \Ima K^{[p]})    \ar[r]^{g_i^p} \ar[d]^{\pi_{i-1}^pU^{[p]}}   & R^n/(\fc_{i}^{[p]} R^n+ \Ima K^{[p]}) \ar[r]^{g_{i+1}^p}  \ar[d]^{\pi_{i}^pU^{[p]}}  & R^n/(\fc_{i+1}^{[p]} R^n+ \Ima K^{[p]}) \ar[r]^>{g_{i+2}^p\quad \quad \quad \quad }  \ar[d]^{\pi_{i+1}^p U^{[p]}} & \dots\\
& \vdots & \vdots & \vdots & \quad \quad .\\
}
\end{equation*}
\item We can now conclude that a generating morphism for $H^i(C({\mathcal U}_{\mathbf{g}},{\mathcal M}))$ is
\[
\frac{(\fc_{i+1} R^n+ \Ima K) : _{R^n}g_{i+1}}{\fc_{i} R^n+ g_iR^n + \Ima K} \stackrel{\pi_i U}{\lra} \frac{((\fc_{i+1} ^{[p]}R^n+ \Ima K^{[p]}) : _{R^n}g_{i+1}^p)}{(\fc_{i}^{[p]} R^n+ g_i^pR^n + \Ima K^{[p]})}.
\]
Unfortunately, this homomorphism, which we here abbreviate by $\gamma : G \lra F(G)$, need not be injective. We again use the ideas of the proof of Lyubeznik \cite[Proposition 2.3]{Lyube97}: for each $i \in \N$, let $\gamma_i : G \lra F^i(G)$ be the composition
    \[
\begin{picture}(300,15)(-150,35)
\put(-120,40){\makebox(0,0){$                   G
$}}
\put(-60,40){\makebox(0,0){$                    F(G)
$}}
\put(-7,40){\makebox(0,0){$                      \cdots
$}}
\put(52,40){\makebox(0,0){$                     F^{i-1}(G)
$}}
\put(120,40){\makebox(0,0){$                    F^{i}(G);
$}}
\put(-110,40){\vector(1,0){36}}
\put(-90,44){\makebox(0,0){$^{             \gamma
}$}}
\put(-47,40){\vector(1,0){29}}
\put(-34,44){\makebox(0,0){$^{             F(\gamma)
}$}}
\put(0,40){\vector(1,0){30}}
\put(71,40){\vector(1,0){32}}
\put(87,44){\makebox(0,0){$^{             F^{i-1}(\gamma)
}$}}
\end{picture}
\]
if $t$ is the smallest integer such that $\Ker \gamma_t = \Ker \gamma_{t+1}$ (and there will be such), then $\Ima \gamma_t =: J$ is a root for $H^i(C({\mathcal U}_{\mathbf{g}},{\mathcal M}))$.
\item We can now use part (iv) to conclude that the desired Lyubeznik number is equal to the $R/\fm$-vector-space dimension of the annihilator
$(0 : _J \fm)$.
\end{enumerate}
\end{disc}

%{Lnos.2add} has yet to be checked.
\begin{ex}
\label{Lnos.2add} In this first example, the procedure described in \ref{Lnos.2} simplifies considerably.

Let $\mathbb{K} := \Z/2\Z$, let $R := \mathbb{K}[X_1,X_2,X_3,X_4,X_5]$, and set $$\fm : = (X_1,X_2,X_3,X_4,X_5)R, \quad \fn : = (X_1,X_2,X_3,X_4)R,$$ \[
\fc := (X_1,X_2,X_3,X_4)R \cap (X_2,X_3,X_5)R \cap (X_1 -X_5, X_2 - X_5, X_3 - X_4)R.
\]
We are going to use our algorithm to show that $\lambda_{0 1}(R_{\fm}/\fc R_{\fm}) = 1$, so that the Lyubeznik table for $R_{\fm}/\fc R_{\fm}$ is not trivial.

Set $L := \Ext^{5-1}_{R}(R/\fc, R)$ and let $$\alpha : L \lra F(L)= F(\Ext^{5-1}_{R}(R/\fc, R)) = \Ext^{5-1}_{R}(F(R/\fc), F(R))= \Ext^{5-1}_{R}(R/\fc^{[2]}, R)$$ be the generating morphism for $H^{5-1}_{\fc}(R) =: {\mathcal M}$ described in \ref{Lnos.2}(i). We compute this morphism to be isomorphic to the $R$-homomorphism $\beta : (R/\fn) \oplus (R/\fm) \lra (R/\fn^{[2]}) \oplus (R/\fm^{[2]})$ induced by multiplication by $U := X_1X_2X_3X_4X_5$. It is not hard to see that this $\beta$ is injective, and therefore a root of $H^{5-1}_{\fc}(R)$. It is clear that $\Ass( (R/\fn) \oplus (R/\fm)) = \{\fm,\fn\}$, and so $g_1 := X_5$ is an $\fm$-filter-regular sequence on $(R/\fn) \oplus (R/\fm)$ and therefore also on ${\mathcal M}$ (by \ref{alg.3}). We continue with the procedure outlined in \ref{Lnos.2}(vi) and calculate that $\fc_0 := 0^{\lowlim} = 0$,  $\fc_1 := (X_5R)^{\lowlim} = 0$,
\[ V := (\fc_1R^2 + \fn\oplus\fm):_{R^2}X_5 = \fn\oplus R \quad \mbox{and} \quad W := \fc_0 R^2 + \fn\oplus\fm = \fn\oplus\fm.
\]
The upshot is that the map $\xi : V/W \lra V^{[2]}/W^{[2]}$ induced by multiplication by $U$ is a root for $H^0_{\fm}(H^{5-4}_{\fc}(R))$. But $\xi$ is isomorphic to the map $R/\fm \lra R/\fm^{[2]}$ induced by multiplication by $U$. We are able to conclude that $\lambda_{0 1}(R_{\fm}/\fc R_{\fm}) = 1$.

Further calculations using our algorithm yielded the Lyubeznik table for $R_{\fm}/\fc R_{\fm}$ as
$
\left[
\begin{array}{cc}
0&1\\
&2\\
\end{array}
\right].
$
\end{ex}

\begin{ex}
\label{Lnos.3} Here we illustrate the operation of our algorithm
by computing in some detail some Lyubeznik numbers of a ring whose characteristic zero counterpart was studied by Alvarez Montaner in
\cite[\S 5]{Alvarez2000}.
%'In some detail' was added.

Let $\mathbb{K}$ be a field and
$R=\mathbb{K}[X_1, X_2, X_3, X_4, X_5, X_6, X_7]$;
write $\mathfrak{m}=(X_1, X_2, X_3, X_4, X_5, X_6, X_7)R$,
and let
$\fc:=(X_1, X_2)R \cap (X_3, X_4)R \cap(X_5, X_6, X_7)R$.
 We now illustrate our algorithm by
using it to compute $\lambda_{3 4}(R_{\fm}/\fc R_{\fm})$ and $\lambda_{4 4}(R_{\fm}/\fc R_{\fm})$
in the case where $\mathbb{K} = \Z/2\Z$.

The first step is to compute a free resolution $\mathbf{F}_\bullet$
of $R/\fc$ and to lift the quotient map
$R/\fc^{[2]} \lra R/\fc$ to a map of
resolutions
$F^1_R(\mathbf{F}_\bullet)  \lra \mathbf{F}_\bullet$.
We next apply $\Hom_R( - , R)$ to this map and compute cohomology to obtain maps $\Ext^\bullet_R(R/\fc , R) \lra \Ext^\bullet_R(R/\fc^{[2]} , R)$. Specifically in this example we obtain a map $\Ext^4_R(R/\fc , R) \lra \Ext^4_R(R/\fc^{[2]} , R)$ given by
$\phi: \Coker K \xrightarrow{U} \Coker K^{[2]}$
where
$$K:=
\left[\begin{array}{llllllllll}
X_7 &X_6 &X_5 &X_2 &X_1 &0   &0   &0   &0   &0   \\
0   &0   &0   &0   &0   &X_7 &X_6 &X_5 &X_4 &X_3 \\
\end{array}
\right]
$$
and
$$
U:=
\left[\begin{array}{ll}
X_1X_2X_5X_6X_7& 0 \\
0               & X_3 X_4 X_5 X_6 X_7 \\
\end{array}
\right].
$$

One can verify that this map is injective and hence a root.

The next step is to find an $\mathfrak{m}$-filter-regular sequence on $\Coker K$. One such sequence begins
\begin{align*}
(& g_1, g_2, g_3, g_4, g_5, \ldots)\\&=\left( X_1+X_2+X_4+X_5+X_7,
X_2+X_3+X_4+X_5+X_6+X_7,
X_3+X_6+X_7,
X_2, X_1,\ldots
\right)
\end{align*}
For our calculations, we only need details of the first five terms in the $\mathfrak{m}$-filter-regular sequence.
Write $\pi_i=g_1 \dots g_i$ for all $i \in \N$.

We calculate that
$$\fc_3:=(g_1, g_2, g_3)^{\lowlim}=(X_2+X_3+X_4+X_5+X_6+X_7, X_1+X_3+X_6)R $$
and
$$\fc_4:=(g_1, g_2, g_3, g_4)^{\lowlim}=(X_1+X_7, X_2+X_4+X_5, X_3+X_6+X_7)R.$$

We now compute a generating
morphism for $H^3(C({\mathcal U}_{\mathbf{g}},H^{7-4}_{\fc} (R)))$ \big($ \cong H^3_{\mathfrak{m}} \left( H^{7-4}_{\fc} (R)\right)$\big) as follows. We calculate that
$$B:=(\fc_4 R^2+\Ima K):_{R^2} g_4= R \oplus ( X_1, X_2+X_4, X_3, X_5, X_6, X_7)R,$$
$$C:=\fc_3 R^2+g_3 R^2+\Ima K=
\mathfrak{m} \oplus (X_1,  X_2+X_4, X_3, X_5, X_6, X_7 )R$$ and
$(C^{[2]} :_{R^2} \pi_3 U) \cap B= C,$
and we conclude that $\lambda_{3 4}(R_{\fm}/\fc R_{\fm})=\dim_{R/\fm} (0 :_{  B/C}\fm)=1$.

To find a generating
morphism for $H^4_{\mathfrak{m}} \left( H^{7-4}_{\fc} (R)\right)$
we calculate that
$$\fc_4=(g_1, g_2, g_3, g_4)^{\lowlim}=( X_1+X_7, X_2+X_4+X_5, X_3+X_6+X_7 )R$$
(as above) and
$$\fc_5:=(g_1, g_2, g_3, g_4, g_5)^{\lowlim}=( X_1+X_7, X_2, X_3+X_6+X_7, X_4+X_5   )R.$$
We compute
$$V:=(\fc_5 R^2+\Ima K):_{R^2} g_5= R^2, \quad
W:=\fc_4 R^2+g_4 R^2+\Ima K=\mathfrak{m} \oplus \mathfrak{m},$$
$$(W^{[2]} :_{R^2} \pi_4 U) \cap V=R \oplus \mathfrak{m}\quad \mbox{and} \quad (W^{[4]} :_{R^2} \pi_4^3 U^{[2]} U)\cap V =R \oplus \mathfrak{m}.$$
Hence
$\lambda_{4 4}(R_{\fm}/\fc R_{\fm})=\dim_{R/\fm} R/\mathfrak{m} =1$.

We have used our algorithm to calculate the Lyubeznik table $\Lambda(R_{\fm}/\fc R_{\fm}):=(\lambda_{i j}(R_{\fm}/\fc R_{\fm}))$ of $R_{\fm}/\fc R_{\fm}$ as
$$
\Lambda(R_{\fm}/\fc R_{\fm})=\left[
\begin{array}{cccccc}
0&0&1&0&0&0\\
&0&0&0&0&0\\
&&0&2&0&0\\
&&&0&1&0\\
&&&&1& 0\\
&&&&&2\\
\end{array}
\right].
$$
Interestingly, this is exactly the same Lyubeznik table as that given by Alvarez Montaner in \cite[\S 5]{Alvarez2000} for
$R_{\fm}/\fc R_{\fm}$ when $\mathbb{K}$ is a field of characteristic $0$.
\end{ex}

\begin{exs}
\label{Lnos.5} We have tested our algorithm in some situations where Lyubeznik numbers are already known. Here are some examples. In them, $R$ denotes the polynomial ring with coefficients in $\Z/2\Z$ in a finite number $n$ of variables which will be denoted by upper case letters with numeric suffices, such as $X_1, X_2, A_1, A_2, \ldots$; $\fm$ denotes the maximal ideal of $R$ generated by all the variables and $\fc$ is an ideal of $R$ contained in $\fm$. We are interested in the Lyubeznik table for $R_{\fm}/\fc R_{\fm}$ in various cases. Notice that this is equal to the Lyubeznik table for $R'/\fc R'$, where $R'$ is the ring of formal power series with coefficients in $\Z/2\Z$ in the same variables used for the construction of $R$.
\begin{enumerate}
\item We have used our algorithm in the case where $n = 6$ and \begin{align*} \fc & =  (X_1X_2X_3)R + (X_1X_2X_4)R + (X_1X_3X_5)R + (X_2X_4X_5)R + (X_3X_4X_5)R\\ & + (X_2X_3X_6)R + (X_1X_4X_6)R + (X_3X_4X_6)R + (X_1X_5X_6)R + (X_2X_5X_6)R\end{align*}
to calculate the Lyubeznik table for $R_{\fm}/\fc R_{\fm}$ as
$
\left[
\begin{array}{cccc}
0&0&1&0\\
&0&0&0\\
&&0&1\\
&&&1
\end{array}
\right],
$
just as reported by Alvarez Montaner and Vahidi in \cite[Example 4.8]{AlvVah14}.
\item When $n = 4$ and $\fc = X_1X_2R + X_2X_3R + X_3X_4R + X_4X_1R $, our calculations using our algorithm give the Lyubeznik table for $R_{\fm}/\fc R_{\fm}$ as
$
\left[
\begin{array}{ccc}
0&1&0\\
&0&0\\
&&2\\
\end{array}
\right],
$
just as reported by Nadi, Rahmati and Eghbali in \cite[Example 4.6]{NadRahEgh2017}.
\item When $n = 5$ and $\fc = (X_1,X_2,X_3)R \cap (X_3,X_4,X_5)R \cap (X_1,X_2,X_3,X_4)R$, our calculations using our algorithm yielded the Lyubeznik table for $R_{\fm}/\fc R_{\fm}$ as
$
\left[
\begin{array}{ccc}
0&1&0\\
&0&0\\
&&2\\
\end{array}
\right],
$
just as reported by Alvarez Montaner and Vahidi in \cite[Example 5.3]{AlvVah14}.
\item When $n = 6$ and $\fc = (X_1,X_2)R \cap (X_3,X_4)R \cap (X_5,X_6)R$, our calculations using our algorithm show that the Lyubeznik table for $R_{\fm}/\fc R_{\fm}$ is
$
\left[
\begin{array}{ccccc}
0&0&1&0&0\\
&0&0&0&0\\
&&0&3&0\\
&&&0&0\\
&&&&3
\end{array}
\right],
$
just as reported by N\'u\~nez-Betancourt, Spiroff and Witt in \cite[Example 2.4 and 5.5]{NunSpiWit19}.

\item We have used our algorithm in the case where $n = 7$ and $$\fc = X_1X_2R + X_2X_3R + X_3X_4R + X_4X_5R + X_5X_6R + X_6X_7R + X_7X_1R$$ to calculate the Lyubeznik table for $R_{\fm}/\fc R_{\fm}$ as
$
\left[
\begin{array}{cccc}
0&0&1&0\\
&0&0&0\\
&&0&1\\
&&&1
\end{array}
\right],
$
just as reported by Nadi, Rahmati and Eghbali in \cite[Example 4.6]{NadRahEgh2017}.
\item In the case where $n=6$ and $\fc$ is generated by  $\{A_1A_2,  B_1B_2,  C_1C_2,  A_1B_1C_1,   A_2B_2C_2  \}$, our algorithm produced $\left[
\begin{array}{cccc}
0&0&1&0\\
&0&0&0\\
&&0&1\\
&&&1
\end{array}
\right]$ as the Lyubeznik table for $R_{\fm}/\fc R_{\fm}$. This is consistent with the results reported by Nematbakhsh in \cite[Example 4.2]{Nemat18}.
\end{enumerate}
\end{exs}

\begin{exs}
\label{Lnos.4} Other situations where we have tested our algorithm against known Lyubeznik numbers include those studied by
De Stefani, Grifo and N\'u\~nez-Betancourt in \cite[Example 4.11]{DeSGriNun17}, and
Singh and Walther in \cite[Example 2.3]{SinWal05}.
Again, our results using our algorithm are consistent with the previously published results.

However, we also also found evidence that the usefulness of our algorithm can depend on the computer(s) available for the calculations. In 2020, we tried the algorithm in the case where $n = 8$ and $\fc$
 is generated by
 \begin{align*}A_1A_2A_3A_4 &,  B_1B_2B_3B_4,  A_1A_3A_4B_1B_2B_3,    A_2A_3A_4B_1B_4,     A_1A_2A_4B_2B_4,\\ & A_1A_2A_3B_3B_4,    A_2A_4B_1B_2B_4,       A_2A_3B_1B_3B_4,     A_1A_2B_2B_3B_4.\end{align*}
 The complexities of the calculations, via our algorithm, of the Lyubeznik table for $R_{\fm}/\fc R_{\fm}$ turned out to be too severe for the  laptop available in the 2020 `lockdown'. However, we had success with this example in 2022 using a more powerful computer.
\end{exs}

\bibliographystyle{amsplain}

\end{document}